\documentclass[a4paper,12pt]{report}
\usepackage{inputenc}
\usepackage[english]{babel}
\usepackage{hyperref}
\usepackage{amsmath,enumerate,amsthm }
\usepackage{amsfonts}
\usepackage{graphics,color,graphicx}
\usepackage{wrapfig}
\usepackage{amssymb}
\usepackage[all]{xy}
\usepackage[nottoc]{tocbibind}
\usepackage{marginnote}
\usepackage{accents}
\usepackage[toc,page]{appendix}

\newcommand{\D}{\mathbb{D}}
\newcommand{\F}{\mathcal{F}}
\newcommand{\X}{\mathcal{X}}
\newcommand{\G}{\mathcal{G}}
\newcommand{\C}{\mathbb{C}}

\newcounter{tmp}
\newtheorem{theorem}{Theorem}[section]
\newtheorem{coro}[theorem]{Corollary}
\newtheorem{pro}[theorem]{Proposition}
\newtheorem{lem}[theorem]{Lemma}
\theoremstyle{definition}
\newtheorem{defi}[theorem]{Definition}

\theoremstyle{remark}
\newtheorem{aff}{Affirmation}[section]
\newtheorem{rem}{Remark}

\begin{document}
\begin{titlepage}
    \begin{center}
        \vspace*{1cm}
        \includegraphics[scale=0.1]{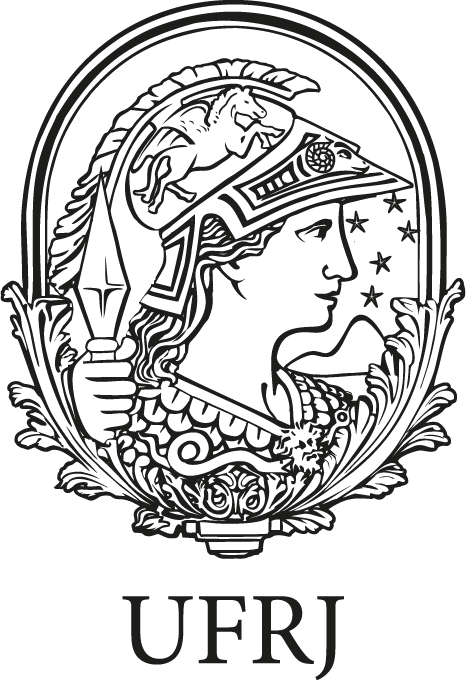}

        \textsc{Mathematics Institute\\
        Federal University of Rio de Janeiro}\\

        \vspace{3 cm}
        \textsc{Ph.D. Thesis}\\
        \vspace{0.5 cm}
        \Large
        \textbf{On first integrals of holomorphic foliations}

        \vfill     
               
        \textsc{Jonny Ardila Ardila}
        
        \vfill

	\normalsize
	\textsc{Brazil}\\
        2016        
    \end{center}
\end{titlepage}
\tableofcontents

\chapter{Groups of germs diffeomorphisms}\label{Chap1}
An important tool in the study of foliations (real and holomorphic) are the holonomy groups two clear examples (among many other) of this affirmation are the 
stability theorems of Reeb (see \cite{livCamNet} chap. IV) and the theorem of existence and uniqueness of first integrals of Mattei-Moussu \cite{M-M}. In the context of holomorphic 
foliation the holonomy groups are just finitely generated groups of germs of diffeomorphisms in $\C^n$ fixing the origin, those groups have been highly studied for many authors and 
important results have been obtained in dimension $1$ and in general dimension.\par      
In the next section, after introduce some definitions and notation, we will mention some of this results that although they are interesting by they own, the 
way how they intervene throughout this work is what transform them in a fundamental piece of this thesis.\par   
Sections two and three are based on the Theorem 3.1 in \cite{Brochero}. The Theorem \ref{BroN} is its generalization to dimension $n>
2$ (as the author points out in \cite{Brochero} ) and Theorem \ref{BroG} is its version for finite generated groups. In Theorems \ref{Ard&Bro1} and \ref{Ard&Bro2} we
make few changes to its hypothesis, maintaining valid the original conclusion, obtaining in this way two new versions of it.\par It is worth to say that only small changes in the 
original proof in \cite{Brochero} are needed to 
demonstrate the previous theorems. Nevertheless, we will write down each one of the proofs in order to make easy to note the difference among them.\par
We end this chapter with some comments of recent results in this topic (see \cite{Reb-Reis,BS-closedorbits}).
\section{Preliminaries}
Let Diff$(\C^n,0)$ be the group of germs of diffeomorphisms at $0$, the germ $G\in\mathrm{Diff}(\C^n,0)$ will be represented by the map $G$ in a domain $U$ where $G(U)$ and
$G^{-1}(U)$ are well defined, and $U$ is an open neighborhood of the origin with compact closure. We will use the following notation, 
\begin{align*}
O_U(x,G)&=\{G^p(x)\,|\,G(x),\dots,G^p(x)\in U\}\cup\\ &\quad\ \{G^{-q}(x)\,|\,G^{-1}(x),\dots,G^{-q}(x)\in U\}\cup\{x\}
\end{align*}
for the $G$-orbit of $x$ in $U$,  $|O_U(x,G)|$ for the number of elements in its $G$-orbit and 
\begin{align*}
\mu_U(x,G)&=\sup\{p>0\,|\,G^p(x)\in O_U(x,G)\}+\\&\quad\,\sup\{p>0\,|\,G^{-p}(x)\in O_U(x,G)\}+1, 
\end{align*}
for the number of iterates of $x$ in $U$. If $\mu_U(x,G)=\infty$ and $|O_U(x,G)|<\infty$ we say that the point $x$ is \emph{periodic} in $U$, if $\mu_U(x,G)$ is finite it 
is equal to $|O_U(x,G)|$. We say that $G$ has \emph{finite orbits} if $|O_U(x,G)|<\infty$ for all $x\in U$ \par 
Regarding the finiteness of groups generated by a germ of diffeomorphisms, Mattei-Moussu give in \cite{M-M} p. 477 the following criteria for the one 
dimensional case. 
\begin{theorem}\label{fund-theo}
A element $G\in\mathrm{Diff}(\C,0)$ is periodic if and only if it has finite orbits.
\end{theorem}
Another prove of this theorem (using P\'erez-Marco's work) is given in \cite{Moussu}.\smallskip\\
It is easy to see that Theorem \ref{fund-theo} is not true in dimension grater than one, but with an additional hypothesis Theorem \ref{Bro1} (which is Theorem 3.1 in 
\cite{Brochero}) attempts to generalize this criteria. The reason we say "attempts" is because the prove presented in \cite{Brochero} is inaccurate, we believe in the result but our 
attempt to prove it did not succeed it, for this reason we give an additional hypothesis that allow us to prove it as we do below in Theorem \ref{BroN}. 
\begin{theorem}[Brochero]\label{Bro1}
Let $G\in\mathrm{Diff}(\C^2,0)$. Then $G$ generates a finite group if and only if, there exists a neighborhood $V$ of $0$ such that
$|O_V(x,G)| < \infty$ for all $x\in V$ and $G$ leaves invariant infinitely many analytic varieties at $0$.
\end{theorem}
In fact, in the previous two theorems we can change the diffeomorphism $G$ by a finite generated group $\mathcal{G}\subset\mathrm{Diff}(\C,0)$ (or $\mathrm{Diff}(\C^2,0)$ 
respectively) according to the second affirmation of Lemma 3.3 in \cite{Fabio} that says:
\begin{lem}\label{lemFab}
Let $\G\subset\mathrm{Diff}(\C^k,0)$ be a finitely generated subgroup. Assume that there is an invariant connected neighborhood $W$ of the origin in $\C^k$ such that each
point $x$ is periodic for each element $G\in\G$. Then $\G$ is a finite group.
\end{lem} 
The following topological lemma (which is a modification of the Lewowicz's Lemma) plays an important role through this chapter.
\begin{lem}\label{firstLem}
Let $M$, $0\in M$, be a complex analytic variety of $\C^n$ and $K$ a connected component of $0$ in $\overline{B}_r(0)\cap M$. Suppose that $f$ is a homeomorphism from $K$ to
$f(K)\subset M$ such
that $f(0) = 0$. Then there exists $x\in\partial K$ such that the number of iterations $f^m (x)\in K$ is infinity.
\end{lem} 
\begin{proof}
 Denote by $\overline{\mu} = \mu|_K$ and $\mu = \mu|_{\accentset{\circ}{K}}$ the number of iteration in $K$ and $\accentset{\circ}{K}$. It is easy to see that $\overline{\mu}$ is 
upper semicontinuous, $\mu$ is under semicontinuous and $\overline{\mu}(x) \geq \mu(x)$ for all $x\in\accentset{\circ}{K}$. Suppose by contradiction that $\overline{\mu}(x)<\infty$ 
for
all $x\in\partial K$, therefore exists $n\in\mathbb{N}$ such that $\overline{\mu}(x) < n$ for all $x\in\partial K$. Let
$A = \{x\in K\,|\,\overline{\mu}(x) < n\}\supset ∂K$ and $B = \{x\in\accentset{\circ}{K}\,|\,\mu(x)\geq n\}\ni 0$ open set,
and $A\cap B=\emptyset$ since $\overline{\mu}(x)\geq\mu(x)$.\\
Using the fact that $K$ is a connected set, there exists $x_0\in K\setminus(A\cup B)$ i.e
$\overline{\mu}(x_0)\geq n> \mu(x_0)$, then the orbit of $x_0$ intersects the border of $K$, which is a contradiction since $\partial K\subset A$ implies $x_0\in A$.
\end{proof}

In our framework Lemma \ref{firstLem} implies: 
\begin{lem}\label{secLem}
Let $G\in\mathrm{Diff}(\C^n,0)$ and $M$ be a $G$-invariant complex analytic variety passing through $0\in\C^n$. It exists a compact, connected, and non-enumerable set $C_M$ such 
that $0\in C_M$ and, for 
all $x\in
C_M$ and $n\in\mathbb{N}$ we have $G^n(x)\in M\cap V$ for a domain $V$ where $G(V)$ and $G^{-1}(V)$ are well defined. 
\end{lem}
\begin{proof}
 Without loss of generality we suppose that $V = \overline{B}_r(0)$. Let $M$ be a $G$-invariant complex analytic variety and $K= M\cap V$ the connected component of
$M\cap V$ in $0$. Let $A_1 = K$, $A_{j+1} = K\cap G^{-1}(A_j)$ and $C_n$ the connected component of $A_n$ in $0$. It is clear, by construction that $A_n$ is the set of points of $K$
with $n$ or more iterates by $G$ in $K$. Moreover, since $A_n$ is compact and $C_n$ is compact and connected, it follows that $C_M =\bigcap_n C_n$ is compact and connected too, and
therefore $C_M = \{0\}$ or $C_M$ is non enumerable.\par
We claim that $C_M\cap\partial K \neq \emptyset$ and then it is non enumerate. In fact, if $C_M\cap\partial K=\emptyset$ then there exists $j$ such that $C_j\cap\partial K =
\emptyset$. Let $B$ be a compact connected neighborhood of $C_j$ such that $(A_j\setminus C_j )\cap B = \emptyset$, therefore for all $x\in\partial B$ we have $\mu_B(x,G) < j$, that
is a contradiction by the Lemma \ref{firstLem}.
\end{proof}
The previous lemma is part of the proof of Theorem \ref{Bro1} in \cite{Brochero}, but due to its importance and constant use throughout the chapter, we decided to write it as an 
independent result. 
\section{Groups of diffeomorphisms in dimension $n$ fixing $0$}
We start by presenting a proof of Theorem \ref{Bro1} in dimension $n$. In this prove we follow the original one just adapting some argument to this case and changing one of the 
hypothesis in order to avoid an imprecision found in the original proof (later on we will discuss this topic). \par\smallskip
The following proposition is the analytic case of Proposition 3.1 in \cite{Brochero}, it is also true in the formal case (the demonstration is the same) and it will be use in 
the proof of the Theorem \ref{BroN}. 
\begin{pro}\label{fingroup}
Let $\G$ be a finite subgroup of $\mathrm{Diff}(\C^n,0)$ then $\G$ is analytic linearizable, and it is isomorphic to a finite subgroup of $Gl(n, \C)$.
\end{pro}
\begin{proof}
  If $\G=\{G_1,\dots,G_r\}$, let $h^{-1}(x)=\sum_j^r(\mathrm{d}G_j)^{-1}_0G_j(x)$, Note that $h$ is a diffeomorphism because $\mathrm{d}h(0)=rI$ and
  \begin{align*}
      h^{-1}\big(G_i(x)\big)&=\sum_j^r(\mathrm{d}G_j)^{-1}_0G_j\big(G_i(x)\big)=(\mathrm{d}G_i)_0\sum_j^r(\mathrm{d}G_i)^{-1}_0(\mathrm{d}G_j)^{-1}_0G_j\big(G_i(x)\big),\\
			    &=(\mathrm{d}G_i)_0\sum_j^r\big((\mathrm{d}G_j)_0(\mathrm{d}G_i)_0\big)^{-1}_0G_j\big(G_i(x)\big),\\
			    &=(\mathrm{d}G_i)_0\sum_j^r\mathrm{d}\big(G_j\circ G_i\big)_0^{-1}G_j\big(G_i(x)\big)=(\mathrm{d}G_i)_0h^{-1}(x).
  \end{align*}
  Thus $h^{-1}\circ G_i\circ h(x)=(\mathrm{d}G_i)_0(x)$. In fact, we obtain a injective groups homomorphism
\begin{align*}
      \G&\overset{\Lambda}{\longrightarrow} Gl(n, \C)\\
		G&\longrightarrow (h^{-1}\circ G\circ h)'(0).\qedhere
\end{align*} 
\end{proof}\noindent
Furthermore, in \cite{Brochero} is proved (after the proposition above) that the group $\Lambda(\G)\subset Gl(n, \C)$ of linear parts of the diffeomorphisms
in $\G$ is diagonalizable.\par
The following theorem is the generalization of Theorem \ref{Bro1} to dimension $n$ but, as we mention above, it was necessary to change one of the hypothesis. To be precise, instead 
of "$G$ leaves invariant infinitely many analytic varieties at $0$" we put "$G$ leaves invariant a non-countable number of hypersurfaces at $0$". In order to clarify this point, 
after Theorem \ref{BroG} we write down the proof of Theorem \ref{Bro1} and we explain why it was necessary for us to make this change.
\begin{theorem}\label{BroN}
Let $G\in\mathrm{Diff}(\C^n,0)$. Then $G$ generates a finite group if and only if, there exists a neighborhood $V$ of $0$ such that
$|O_V(x,G)| < \infty$ for all $x\in V$ and $G$ leaves invariant a non-countable number of hypersurfaces at $0$.
\end{theorem}
\begin{proof}
$(\Rightarrow)$ If the group generated by $G$ is $\G=\{G,G^2,\dots,G^r\}$ obviously for all $x$ in a neighborhood $V$ where $G^i$ is defined for all $i$, we have that 
$O_V(x,\G)$ is finite, in fact $O_V(x,\G)=\{G(x),\dots,G^r(x)\}$.\par Now, consider as in Proposition \ref{fingroup} $h^{-1}(x)=\sum_j^r(\mathrm{d}G^j)_0^{-1}G^j(x)$ which 
is such that $h^{-1}\circ G^i\circ h(x)=(\mathrm{d}G^i)_0(x)$ for all $i$ where $(\mathrm{d}G^i)_0^{n_i}=I$ for some $n_i$ this implies that $(\mathrm{d}G^ i)_0$ is 
diagonalizable then suppose that it is diagonal, in general $h$ can be defined as a diffeomorphism who also diagonalize the group because in this case the group is cyclic then the 
linear parts are simultaneously diagonalizable, thus in the definition of $h^{-1}$ change $(\mathrm{d}G^j)_0$ by $P^{-1}(\mathrm{d}G^j)_0 P$ where $P$ is the matrix who diagonalize 
the group of linear parts and is easy to see that the prove of Proposition \ref{fingroup} works, with this in main define  
\begin{equation}\label{varinvar}
  M_c=\big\{h(x)\in V\,\big|\,c_1x_1^m+\dots+c_nx_n^{m}=0\big\}, 
\end{equation}
where $m=n_1\cdots n_r$ and $c=(c_1,\dots,c_n) $. $M_c$ is a $\G$-invariant complex
analytic variety of dimension $n-1$ for each $c\in\C^n$. In order to see this, take $y\in M_c$ who by definition is equal to $h(x)$ for some $x\in V$ satisfying \ref{varinvar} then 
we have to prove that $G^i(y)\in M_c$ for $i=1,\dots, r$,
\begin{align*}
   G^i(y)&=G^i(h(x))=h\big(h^{-1}\circ G^i\circ h(x)\big),\\
         &=h\big((\mathrm{d}G^i)_0x\big),
\end{align*}
 and using that $(\mathrm{d}G^i)_0$ is diagonal we have (in multi index notation)
 \[\big((\mathrm{d}G^i)_0x\big)^m={(\mathrm{d}G^i)}^m_0x^m=x^m,\]
 therefore, if $y=h(x)\in M_c$ then $ G^i(y)=h\big((\mathrm{d}G^i)_0x\big)\in M_c$.\medskip\\
$(\Leftarrow)$ Consider $M=\C^n$ in Lemma \ref{secLem} then, $C = C_{\C^n}$ is the compact, connected and non-enumerable set of points in $V$ such that $\mu_V(x,G)=\infty$ and 
therefore every point in $C$ is periodic. If we denote $D_m =\bigcup \{x\in C\,|\,G^{m!}(x)=x\}$, it is clear that $D_m$ is a close set and $D_m\subset D_{m+1}$, moreover $C =\bigcup 
D_m$. Fix a $m\in\mathbb{N}$ and consider $F=G^{m!}$ where it is well defined, observe that $C$ is in the domain $U$ of $F$ and take $L=\{x\in U|\,F(x) = x\}$. Since $L$ is a 
complex analytic variety of $V$ then it can be written as a finite union of complex analytic varieties of dimension ranging from $1$ to $n$, but even if all where of dimension 
$n-1$, using 
Lemma \ref{secLem} with every invariant analytic variety $W$ we conclude that there are non-enumerable $C_W\subset C$ not contained in the decomposition of $L$, as the $m$ fixed 
is arbitrary and $C$ is a enumerable union, it can be deduced that there exist a $m$ such that $L$ is of dimension $n$, and it follows that $G^{m!}(x) 
= x$ for all $x\in U$ by the identity theorem (see \cite{gunning1} pag 5), hence the group generated by $G$ is finite.
\end{proof}\noindent
The version of the previous theorem for groups of diffeomorphisms finitely generated is immediate,
\begin{theorem}\label{BroG}
Let $\G=\langle\{G_1,\dots,G_m\}\rangle\prec \mathrm{Diff}(\C^n,0)$ be a finitely generated subgroup of diffeomorphisms. Then $\G$ is finite if and only if, there exists a 
neighborhood $V$ of $0$ such
that $|O_V(x,\G)| <
\infty$ for all $x\in V$ and each $G_i$ leaves invariant a non-countable number of hypersurfaces at $0$.
\end{theorem}
\begin{proof}
$(\Rightarrow)$ This part is the same as the previous theorem, note that in the hypothesis each generator of the group leaves invariant infinitely many analytic varieties, then we
can apply the same construction for each one.\smallskip\\
$(\Leftarrow)$ Using Theorem \ref{BroN} we have that every element in $\G$ has finite order and, $\G$ is finite generated so we can apply Lemma \ref{lemFab} concluding that
$\G$ is finite.  
\end{proof}
The following is the proof of Theorem \ref{Bro1} as can be seen in \cite{Brochero} page 7. 
\begin{proof}[Proof of Theorem \ref{Bro1}]
 $(\Rightarrow)$ Let $N = \#\langle F\rangle$ and $h\in \mathrm{Diff}(\C^2,0)$ such that $h\circ F\circ h^{-1}(x, y)=(\lambda_1 x,\lambda_2y)$ where $\lambda_1^N=\lambda_2^N=1$. It 
is clear than $|O(x,F)|\leq N$ for all $x$ in the domain of $F$ , and $M_c =\{h(x, y)\,|\, x^N-cy^N = 0\}$ is a complex analytic variety invariant by $F$ for all $c\in\C$.\medskip\\
$(\Leftarrow)$ Consider Lemma \ref{secLem} with $M=\C^2$, then $C = C_{\C^2}$ is a set of point with infinite orbits in a domain $V = \overline{B}_r(0)$ where $F$ and $F^{-1}$ are 
well defined and therefore every point in $C$ is periodic. If we denote $D_n = \{x\in C\,|\,F^{n!}(x) = x\}$, it is clear that $D_n$ is a closed set and $D_n \subset D_{n+1}$, 
moreover $C =\cup_{n=1} D_n$, {\color{red} then exists $n\in N$ such that $C = D_n$}. Let $G = F^{n!}$ where it is well defined, observe that $C$ is in the domain $U$ of $G$ and 
$C\subset \{x\in U\,|\,G(x) = x\} = L$. Since $L$ is a complex analytic variety of $U$ that contain $C$ then its dimension is $1$ or $2$. The case dim$\,L = 1$ is impossible because 
$C_M\subset C\subset  L$ for all $M$ analytic variety $F$-invariant, contradicting that fact that $\mathcal{O}_2$ is Noetherian ring. In the case dim$\,L = 2$ 
follows that $F^{n!}(x) = x$ for all $x\in U$, therefore $\langle F\rangle$ i is
finite. 
\end{proof}
The problem with the prove above is in the affirmation:
\begin{center}
 "\dots, then exists $n\in N$ such that $C = D_n$".  
\end{center}
which is not always true because the sets $D_n$ may have empty interior, in fact if one of them happens to have interior the proof ends by the Identity Theorem. Another way of see 
the problem with this affirmation is to note that the increasing sequence of analytic sets $D_n \subset D_{n+1}$ generates a decreasing sequence of ideals, and even in Noetherian 
rings (as $\mathcal{O}_n$) decreasing sequences of ideals do not always stabilize, they do when they are prime which is equivalent to the $D_n$ be irreducible (see \cite{gunning2} 
pag. 15). Now, if they are irreducible and of dimension $1$ all of them are the same one and the set $C$ consist of a single analytic curve which contradicts that by hypothesis there 
are infinitely many $G$-invariant analytic varieties at $0$, and we are done. It would remain the case when the sequence of ideals does not stabilize.\\ We could not get a different 
proof of the statement \ref{Bro1} and his importance in our work force us to change the hypothesis as you can see in Theorem \ref{BroN}.\par\smallskip
We close this section noting that Theorem \ref{BroN} is valid, as the author \cite{Brochero} mentions, if we consider analytic 
varieties of complex dimension $1$ in \emph{general position} instead of hypersurfaces,
\begin{defi}\label{genpos1}
 We say that infinitely many analytic varieties of complex dimension $1$ are in \emph{general position} if they are not contained in finitely many analytic
varieties of complex dimension $n-1$.
\end{defi}
The only change in the proof is in the "if'' part where is necessary one more step, note that choosing $n-1$ constants $c$ linearly independent, the 
intersection of the corresponding $M_c$ has a component of dimension $1$ passing through $0$. In this 
way we can obtain a non-countable number in general position. The reason why we state the theorem in terms of hypersurfaces is because is more natural and it does not require to add 
more conditions. However, it can be useful to think in dimension one as we see next. 
\begin{theorem}\label{Ard&Bro3}
  Let $G\in\mathrm{Diff}(\C^n,0)$. The group generated by $G$ is finite if and only if, there exist a neighborhood $V$ of $0$ such that $|O_V(x,G)| <
\infty$ for all $x\in V$, and $G$ leaves invariant a non-countable number of analytic varieties of complex dimension $1$, in general position, arbitrarily close to $0$, and each one 
intersecting the set $C=C_{\C^n}$ defined as in Lemma \ref{secLem}.
\end{theorem}
\begin{proof}
 Following the proof of Theorem \ref{BroN}, $C = C_{\C^n}$ is the compact, connected and non-enumerable set of points in $V$ such that $\mu_V(x,G)=\infty$ and 
therefore every point in $C$ is periodic. If we denote $D_m =\bigcup \{x\in C\,|\,G^{m!}(x)=x\}$, it is clear that $D_m$ is a close set and $D_m\subset D_{m+1}$, moreover $C =\bigcup 
D_m$. Fix a $m\in\mathbb{N}$ and consider $F=G^{m!}$ where it is well defined, observe that $C$ is in the domain $U$ of $F$ and take $L=\{x\in U|\,F(x) = x\}$. If some invariant 
analytic variety $W$ intersects $C$ in a periodic point $q\in U$ of order $k$ then, as in the previous proofs, Lemma \ref{secLem} can be applied to the map $G^k$ in some neighborhood 
of $q$ contained in $W$ and we obtaining a compact, connected and non-enumerable set $C_W\subset W$ which is fixed for some iterate of $G^k$ (see \ref{C_W-fixed}), each one of those 
$C_W$ belongs to $C$. Now, since $L$ is a 
complex analytic variety of $V$ then it can be written as a finite union of complex analytic varieties of dimension ranging from $1$ to $n$, but even if all where of dimension 
$n-1$ there are non-enumerable $C_W\subset C$ not contained in the decomposition of $L$, as the $m$ fixed 
is arbitrary and $C$ is a enumerable union, it can be deduced that there exist a $m$ such that $L$ is of dimension $n$, and it follows that $G^{m!}(x) 
= x$ for all $x\in U$ by the identity theorem (see \cite{gunning1} pag 5), hence the group generated by $G$ is finite.  
\end{proof}
\begin{rem}\label{C_W-fixed}
 The reason why $C_W\subset W$ is fixed for some iterate of $G^k$ is because $C_W$ is the set of $G^k$-periodic points and it is non-enumerable then there are infinitely many of some 
order $k'=km$, for $m\in\mathbb{N}$, and they accumulate by compactness. The dimension of $W$ is one hence the Identity Theorem implies that $C_W$ is $G^{k'}$-fixed.
\end{rem}

\section{Conditions over the set of periodic points}
The second part of the proof of Theorem \ref{BroN} make us think that what we really need is a sufficient amount of periodic points, but even in dimension one, infinitely many 
accumulating $0$ is not enough. To be precise, according to Perez-Marco in \cite{PM} is possible to construct map germs in Diff$(\C,0)$ exhibiting a sequence of periodic points 
converging to $0$ and not linearizable, obviously the order of the points in that sequence goes to infinity because if some subsequence has bounded order by some $m$ then after $m!$ 
iterates the function has a sequence of fixed points accumulating $0$ and by the identity theorem that iteration is the identity then the map is periodic. However, in dimension 
greater than 
$1$ to have a convergent sequence of fixed points is not enough to guarantee that a map is the identity that is why we asked for a dense set of periodic points while keeping the 
bound 
over the order.
\begin{theorem}\label{Ard&Bro1}
Let $G\in\mathrm{Diff}(\C^n,0)$. The group generated by $G$ is finite if and only if, it exists $m\in\mathbb{N}$ such that for an arbitrary neighborhood of $0$ the set of periodic
orbits of period at most $m$ is dense.
\end{theorem}
\begin{proof}
($\Rightarrow$) Suppose $\langle G\rangle=\{id,\dots,G^{r-1}\}$ for $r\in\mathbb{N}$ and $G$ well defined in a neighborhood $V$ of $0$. Consider
$U$ the connected component of $V\cap G^{-1}(V)\cap\dots\cap G^{r-1}(V)$ at $0$ then every point in $U$, which is an open set, is periodic.\smallskip\\ 
($\Leftarrow$) Consider $F=G^{m!}$ defined in some neighborhood $U$ of $0$ and  $L=\{x\in U\,|\, F(x)=x\}$. Since $L$ is a complex analytic variety of $U$ then it can be written as a
finite union of analytic varieties of dimension ranging from $1$ to $n$, it can not be $0$ because it contains infinite many points accumulating $0$. However, the union of finitely 
many
analytic varieties, even if all of them are of dimension $n-1$, can not contain a dense set of points accumulating $0$. Therefore $\mathrm{dim}\,L=n$ and we have that $G^{m!}(x)=x$
for
all $x\in U$ and we are done. 
\end{proof}
The following theorem shows that we do not need a dense set of periodic points if we have infinitely many let us say "well located".
\begin{theorem}\label{Ard&Bro2}
Let $G\in\mathrm{Diff}(\C^n,0)$. The group generated by $G$ is finite if and only if, it exists $m\in\mathbb{N}$ such that for an arbitrary neighborhood of $0$, $G$ leaves invariant
infinitely many analytic varieties of complex dimension $1$, in general position and each one having a convergent sequence of periodic points of order at most $m$.
\end{theorem}
\begin{proof}
 ($\Rightarrow$) The same as Theorem \ref{BroN}. And we obtain infinitely many analytic varieties of complex dimension $1$ passing through $0$, and the periodicity of the group 
implies that every point on them is periodic of same order. \smallskip\\ 
 ($\Leftarrow$) First, take $F=G^{m!}$ defined in some neighborhood $U$ of $0$, with $m$ as in the statement, and take a $G$-invariant analytic variety $M$ in $U$, by hypothesis $M$ 
has a convergence sequence of periodic points of order at most $m$ converging to some point $p\in M$. We can apply Lemma \ref{secLem} taking $F$ as the map, $M$ the $F$-invariant 
complex analytic variety, $q$ the $F$-fixed point and $K_q$ the connected component of $M$ containing $q$, then there exist a $C_M$ (compact, connected and non-enumerable) 
containing $q$ and a sequence of $F$-fixed points converging to it, by the identity theorem (the one dimensional version because we are restricted to $M$) $K_q$ is formed by 
$F$-fixed points. Now, define $L=\{x\in U\,|\,F(x)=x\}$ which is a complex analytic variety in $U$ then it can be written as a finite union of analytic varieties of dimension ranging 
from $1$ to $n$, as before it can not be $0$ because it contains infinite many points accumulating $0$ (here we are using the hypothesis about the arbitrariness of the 
neighborhoods). 
The case $\mathrm{dim}\,L=1$ is impossible, in order of see this consider $M$ and $q$ as before and note by $L_q$ the irreducible component of $L$ containing $q$. Hence $L_q$ and 
$K_q$ are complex analytic varieties of dimension one equal in a set with an accumulation point then they are the same. The same argument can by applied infinitely many
times and as in the proof of Theorem \ref{BroN}, even if all the irreducible components where of dimension $n-1$ by hypothesis there are still infinitely many not contained in them 
therefore this is impossible. The
remaining case
is dim$\,L = n$ and it follows that $G^{m!}(x)=x$ for all $x\in U$ and we are done. 
\end{proof}
If in Theorem \ref{Ard&Bro2} we make the analytic varieties pass through $0$, we get as a corollary a version of Theorem \ref{BroN} changing the finite many orbits hypothesis by 
periodic points of bounded order accumulating $0$.
\begin{coro}\label{A&B2Cor1}
Let $G\in\mathrm{Diff}(\C^n,0)$. The group generated by $G$ is finite if and only if, it exists $m\in\mathbb{N}$ such that $G$ leaves invariant infinitely many analytic varieties of
complex dimension $1$, in general position and each one having a sequence of periodic points of order at most $m$, that accumulates $0$.
\end{coro}
\section{Advances found in the literature}
The final part of this chapter is dedicated to show some generalizations of Theorem \ref{fund-theo} existent in recent works, their proofs can be found in the referenced 
articles\par\smallskip
The first one we mention is taken from \cite{Reb-Reis} ,  
\begin{theorem}\label{rebreisA}
 Let $\G\subset\mathrm{Diff}(\C^n,0)$ be a finitely generated pseudogroup on a small neighborhood of the origin in $\C^n$. Given $G\in\G$, let Dom$(G)$ denote the domain of 
definition of $G$ as element of the pseudogroup in question. Suppose that for every $G\in\G$ and $p\in\mathrm{Dom}(G)$ satisfying $G(p) = p$, one of the following holds: either $p$ 
is an isolated fixed point of $G$ or $G$ coincides with the identity on a neighborhood of $p$. Then the pseudogroup $\G$ has finite orbits on a neighborhood of the origin if and only 
if $\G$ itself is finite.
\end{theorem}
This theorem is consequence of the following proposition (Proposition 4. in \cite{Reb-Reis}) and an argument like Lemma \ref{lemFab}.
\begin{pro}
 Suppose that $\G\subset\mathrm{Diff}(\C^n,0)$ is a group satisfying the condition of isolated fixed points of Theorem \ref{rebreisA}. Let $G$ be an element of $\G$ and assume
that $G$ has only finite orbits. Then $G$ is periodic.
\end{pro}
As the authors observe, this proposition is obtained repeating the proof of Theorem \ref{fund-theo} in \cite{M-M} p. 477 and noting that the \emph{isolated fixed points condition} 
replace the argument that in dimension one is consequence of the Identity Theorem.\medskip\par
The next generalization of Theorem \ref{fund-theo} moves in another direction, instead of the dimension it deals with the hypothesis of "all orbits be finite" analyzing 
the case where a diffeomorphism has a positive measure sets of closed orbits. This result can be found in \cite{BS-closedorbits} and in its proof is used the work of 
Perez-Marco (\cite{PM,PM1,PM2}).\\\medskip
We need to introduce first some notation:\par\medskip 
Expand a germ of a complex diffeomorphism $f$ at the origin $0\in\C$ as
\[f(z) = e^{2\pi i\lambda}z + a_{k+1}z^{k+1} +\dots,\]
The multiplier $f'(0) =  e^{2\pi i\lambda}$ does not depend on the coordinate system. We shall say that the germ $f\in\mathrm{Diff}(\C, 0)$ is \emph{non-resonant} if 
$\lambda\in\C\setminus\mathbb{Q}$.
\begin{defi}
 A map germ  $f\in\mathrm{Diff}(\C, 0)$ is called a \emph{Cremer map germ} if it is non-linearizable and non-resonant.   
\end{defi}
Cremer gave the first proof of the existence of a such map in \cite{Cremermap}.
\begin{defi}
 We call \emph{(PCO) Cremer map germ} to a Cremer map germ, such that its representatives exhibit positive measure sets of closed orbits, in arbitrarily small neighborhoods of the 
origin.
\end{defi}
\begin{lem}
 Let $\G\subset\mathrm{Diff}(\C,0)$ be a finitely generated subgroup with the (PCO) property. Then either $\G$ is a cyclic finite (resonant) group or it is an abelian formally 
linearizable group, containing some (PCO) Cremer diffeomorphism.
\end{lem}
\chapter{Groups of formal diffeomorphisms and formal series}\label{Chap2}
This chapter is devoted to the study of formal difeomorphisms and formal series. Here we obtain some useful properties for our upcoming work.
\section{Preliminaries}
Let us introduce some standard notation, denote the ring of \emph{formal series} on $(\C^n,0)$ by $\hat{\mathcal{O}}_n$ and the group of \emph{formal diffeomorphisms} of 
$(\C^n,0)$ by $\mathrm{\widehat{Diff}}(\C^n,0)$. The  convergent versions of the previous sets are, the ring of \emph{germs of holomorphic functions} on $(\C^n , 0)$ denoted by 
$\mathcal{O}_n$, its maximal ideal denoted by $\mathcal{M}_n$ and the group of \emph{diffeomorphisms} of $(\C^n,0)$ by $\mathrm{Diff}(\C^n,0)$.\medskip\\ The first 
step is to study the properties we can get from the relationship $\hat{f}\circ \hat{G}=\hat{f}$ where $\hat{f}\in\hat{\mathcal{O}}_n$ and $\hat{G}\in\mathrm{\widehat{Diff}}(\C^n,0)$ 
in this case we say that $\hat{G}$ \emph{leaves} $\hat{f}$ \emph{invariant}, as we state in propositions \ref{pro1} and \ref{pro2} this relationship characterizes both maps. Our work 
will guarantee that we only need to analyze the case where $\hat{G}$ is linearizable. \par\smallskip We start with the following definitions:
\begin{defi}
Let $\Lambda\in\C^n$. We say that a multi-index $Q=(q_1,\dots,q_n)\in\mathbb{N}^n$ with $|Q|=q_1+\dots+q_n\geq 1$, gives a \emph{multiplicative resonant 
relation for $\Lambda$}  if \[\Lambda^Q:=\lambda_1^{q_1}\cdots\lambda_n^{q_n}=1,\]
and if exist a $Q$ giving this property we say that $\Lambda$ is \emph{multiplicative resonant}. 
\end{defi}
Observe that this definition is a particular case of the usual definition of multiplicative resonant that can be seen for example in \cite{Arnold} pp. 192-193, where you can see also 
that the existence of this kind of resonances are the obstruction to formal linearization. Latest results in this topic can be found in \cite{Raissy}.   
\begin{defi}
We shall say that a monomial $x^Q:=x_1^{q_1}\cdots x_n^{q_n}$ is \emph{resonant with respect to} $\Lambda=(\lambda_1,\dots,\lambda_n)\in\C^n$ \big(or simply
\emph{$(\lambda_1,\dots,\lambda_n)$-resonant}\big) if $|Q|\geq 1$ and $\Lambda^Q = 1$.
\end{defi}
\subsection{Formal chain rule}\label{chainrule}
The aim of this paragraph is to show that the Chain Rule holds in the formal case.
\begin{lem}
Let $\hat{F}\in\hat{\mathcal{O}}_n$ and $\hat{G}\in\widehat{\mathrm{Diff}}(\C^n,0)$ be given. Then
\[\mathrm{d}(\hat{F}\circ \hat{G})=\mathrm{d}\hat{F}\cdot\mathrm{d}\hat{G}.\]
\end{lem}
\begin{proof}
We start with $n=1$, let $\hat{f}\in \hat{\mathcal{O}}_1$ given by $\hat{f}(x)=\sum_{i=1}^{\infty}a_ix^i$, define $f_n\in\mathcal{O}_1$ by  $f_n(x)=\sum_{i=1}^{n}a_ix^i$ and take
$g\in\mathcal{O}_1$. We want to show that $\textrm{d}(\hat{f}\circ g)=\textrm{d}\hat{f}_g\,\textrm{d}g$.\par
We already have that $\textrm{d}(f_n\circ g)=(\textrm{d}f_n)_g\,\textrm{d}g$, because they are holomorphic functions, also by the definition of the derivative of a formal series,
we have $\lim_{n\to\infty}\mathrm{d}f_n=\mathrm{d}\hat{f}$. Therefore, what we need to justify is that $\lim_{n\to\infty}(\mathrm{d}f_n)_g=(\mathrm{d}\hat{f})_g$ and
$\lim_{n\to\infty}\mathrm{d}(f_n\circ g)=\mathrm{d}(\hat{f}\circ g)$, both are consequence of the equality   $\lim_{n\to\infty}f_n\circ g=\hat{f}\circ g$ and for this, think in the
coefficient $c_k$ of $x^k$ in $\hat{f}\circ g(x)=\sum^{\infty}_{i=1}c_ix^i=\sum_{i=1}^{\infty}a_i(\sum_{j=1}^{\infty}b_jx^j)^i$, where $g(x)=\sum_{j=1}^{\infty}b_jx^j$. This 
coefficient is formed after algebraic
computation by some of the coefficients in $\sum_{i=1}^{k}a_i(\sum_{j=1}^{k}b_jx^j)^i$, indeed after $i,j=k$ all the elements in $\sum_{i=1}^{\infty}a_i(\sum_{j=1}^{\infty}b_jx^j)^i$
are of order greater than $k$, thus the same coefficients of $x^k$ belongs to both sides of
$\lim_{n\to\infty}f_n\circ g=\hat{f}\circ g$. \par Hence \[\textrm{d}(\hat{f}\circ g)=\textrm{d}\hat{f}_g\,\textrm{d}g,\]
as we wanted.\par
Consider now $g\in\mathcal{O}_2$ and the same $\hat{f}$ that before. In this case the chain rule is consequence of the previous one, because if we fix one of the variables for example
$y=y_0$, then $g(\cdot,y_0)\in\mathcal{O}_1$ and $\frac{\partial}{\partial x}(\hat{f}\circ g)=\textrm{d}\hat{f}_{g(x,y_0)}\,\frac{\partial}{\partial x} g|_{(x,y_0)}$ by the
previous case.\par
The two dimensional case works in a similar way, just take $\hat{F}\in \hat{\mathcal{O}}_2$ and $G(x,y)=(g_1(x,y),g_2(x,y))$ given by $\hat{F}(x)=\sum_{I}a_Ix^iy^j$ and 
$g_1,g_2\in\mathcal{O}_2$, then we
have.
\begin{align*}
\hat{F}\circ G(x,y)&=\sum_{I}a_I\big(g_1(x,y)\big)^i\big(g_2(x,y)\big)^j,\\
&=\sum_i(g_1(x,y))^i\Big(\sum_j a_{i,j}\big(g_2(x,y)\big)^j\Big),\text{ note }\hat{F}_i(x)=\sum_ja_{i,j}x^j,\\
&=\sum_i(g_1(x,y))^i\hat{F}_i\big(g_2(x,y)\big).
\end{align*}
So, $\hat{F}\circ G$ can be written as a sum of products of two formal series $(g_1(x,y))^i$ and $\hat{F}_i\big(g_2(x,y)\big)$, whose derivatives are known by the previous case.
Now note that $\hat{F}\circ G$ is a formal series then is derivation is made term by term, and in the previous paragraph we only rearrange those terms, thus
\begin{align*}
\frac{\partial}{\partial x}\big(\hat{F}\circ G\big)(x,y)&=\sum_{i}\frac{\partial}{\partial x}\Big((g_1(x,y))^i\hat{F}_i\big(g_2(x,y)\big)\Big),\\
&=\sum_i\Big(ig_1^{i-1}\frac{\partial g_1}{\partial x}\hat{F}_i(g_2)+g_1^i\frac{\partial\hat{F}_i}{\partial x}\Big|_{g_2}\frac{\partial g_2}{\partial x}\Big)(x,y),
\end{align*}
\begin{align*}
&=\sum_i\Big(ig_1^{i-1}\frac{\partial g_1}{\partial x}\sum_j a_{i,j}g_2^j+g_1^i\big(\sum_j ja_{i,j}g_2^{j-1}\big)\frac{\partial g_2}{\partial x}\Big)(x,y),\\
&=\sum_{i,j}\Big(ia_{i,j}\big(g_1(x,y)\big)^{i-1}\big(g_2(x,y)\big)^j\frac{\partial g_1}{\partial x}+ja_{i,j}\big(g_1(x,y)\big)^i\big(g_2(x,y)\big)^{j-1}\Big)\frac{\partial
g_2}{\partial x},\\
&=\frac{\partial \hat{F}}{\partial x}\Big|_G\frac{\partial G}{\partial x}(x,y).
\end{align*}
Now consider $\hat{f},\hat{g}\in\hat{\mathcal{O}}_1$, by the previous step $\mathrm{d}(\hat{f}\circ g_n)=\mathrm{d}\hat{f}_{g_n}\mathrm{d}g_n$ where $g_n$ is the truncated series, 
and the chain rule is consequence of $\lim_{n\to\infty}\hat{f}\circ g_n=\hat{f}\circ g$, as before just note that the coefficient of $x^r$ of $\hat{f}\circ g$ appear in $\hat{f}\circ 
g_n$ for all $n>N$ for some $N$. The case $\hat{f}\in\hat{\mathcal{O}}_2$, $\hat{G}\in\widehat{\mathrm{Diff}}(\C^2,0)$ is the same as above.\par
In conclusion, for the case $\hat{F}\in\hat{\mathcal{O}}_2$ and $\hat{G}\in\widehat{\mathrm{Diff}}(\C^2,0)$ the chain rule, $\mathrm{d}(\hat{F}\circ 
\hat{G})=\mathrm{d}\hat{F}\cdot\mathrm{d}\hat{G}$, holds and the process above is easily 
generalized to greater dimension.
\end{proof}
\section{Invariance relationship}
Let us motivate the following proposition with the one dimensional case, take $G(x)=ax$ with $a\in\C\setminus_0$ and $\hat{f}$ the formal series $\hat{f}(x)=\sum_{i\geq 1}a_ix^i$, 
suppose that $\hat{f}\circ G=\hat{f}$ and that $\hat{f}$ is not a power, meaning by this that if $\hat{f}=f_1^{p_1}\cdots f_r^{p_r}$ where $f_1,\dots f_r$ are the $r$ different 
irreducible components of $\hat{f}$ then gcd$(p_1,\dots,p_r)=1$. 
\[\hat{f}(x)=\sum_{i\geq 1}a_ix^i=\hat{f}\circ G(x)=\sum_{i\geq 1}a_i(ax)^i\]  
which implies $a_ia^i=a_i$ for all $i=1,2\dots$, if $\hat{f}\not\equiv 0$ there is a $a_\nu\neq 0$ so, $a^\nu=1$ (i.e. $a$ is a root of the unity) then $a_i=0$ if $i\neq m\nu$ where 
$n\in\mathbb{Z}^+$. In conclusion for this case 
\[G(x)=e^{2\pi i/\nu}x\text{ and } \hat{f}(x)=\hat{l}(x^\nu)\text{ where }\hat{l}\in\hat{\mathcal{O}}_1,\]
if $\hat{f}$ is not a power $\hat{l}$ is invertible i.e., $\hat{l}'(0)\neq 0$ and we have that $(\hat{l}^{-1}\circ\hat{f})(x)=x^\nu$. Therefore, if a formal series 
$\hat{f}$ is invariant by a rotation, there exist and invertible formal series $\hat{l}$ such that $\hat{l}^{-1}\circ\hat{f}$ is holomorphic. Now we explain why $\hat{l}$ is 
invertible, suppose that $\hat{l}(x)= a_px^p+a_{p+1}x^{p+1}+\cdots$ where $p>1$ and $a_p\neq0$ then
\begin{align*}
  \hat{f}(x)&=\hat{l}(x^\nu)= a_px^{p\nu}+a_{p+1}x^{(p+1)\nu}+\cdots,\\
            &=x^{p\nu}(a_p+a_{p+1}x^\nu+\cdots),\\
	 &=\big(g(x^{\nu})\big)^{p},\quad\text{where}\quad g(x)=x(a_p+a_{p+1}x+\cdots)^{1/p}
\end{align*}
as $a_p$ is not $0$, $g$ is well defined and this contradicts the fact that $\hat{f}$ is not a power.\\
The part above is a portion of the Proposition 1.2. in \cite{M-M} and our intention is to generalize it to arbitrary dimension. In order to do that we start with,
\begin{pro}\label{pro1}
Let $\hat{f}\in\hat{\mathcal{O}}_n$ and $\hat{G}\in\mathrm{\widehat{Diff}}(\C^n,0)$ formally linearizable such that $\hat{G}$ leaves $\hat{f}$ invariant. If the linear part of 
$\hat{G}$ is a diagonal matrix, $$\mathrm{d}\hat{G}_0=\mathrm{diag}(\lambda_1,\dots,\lambda_n),$$ then its elements are multiplicative resonant and, $\hat{f}$ after a formal change 
of coordinates is the 
sum of only $(\lambda_1,\dots,\lambda_n)$-resonant monomials.
\end{pro}
\begin{proof}
Start with the linear case taking $G(x)=Ax$ and $\hat{f}(x)=\sum_{|I|\geq 1}a_Ix$,  where $A$ is a non-singular, diagonal 
$n\times n$ matrix and $x=(x_1\dots, x_n)$,  \par
\[G(x_1,\dots,x_n)=(\lambda_1x_1,\dots,\lambda_nx_n),\]
Thus, \[\hat{f}\circ G(x_1,\dots,x_n)=\sum_{|I|\geq 1}a_I(\lambda_1 x_1)^{i_1}\cdots (\lambda_nx_n)^{i_n}=\sum_{|I|\geq 1}a_Ix_1^{i_1}\cdots x_n^{i_n},\] which means that
\[\lambda_1^{i_1}\cdots\lambda_n^{i_n}=1,\text{ for all }I\text{ such that }a_I\neq 0,\]
i.e. $\lambda=(\lambda_1,\dots,\lambda_n)$ is multiplicative resonant.
If $\hat{f}\not\equiv 0$ it is formed only by resonant monomials, furthermore there exist at most $n$ independent (as a vectors) $n$-tuples 
$I=(i_1,\dots,i_n)\in\mathbb{N}^n\setminus_0$ such
that $\lambda_1^{i_1}\cdots\lambda_n^{i_n}=1$.\par In case we have $n$ independent $n$-tuples, all $\lambda_i$'s are roots of the unity as we can see taking logarithm in each one of
the $n$ equalities $\lambda_1^{i_{1\,1}}\cdots\lambda_n^{i_{1,n}}=1$ and solving a linear system like the following 
\[\begin{bmatrix}
i_{1\,1}&\dots&i_{1,n}\\
\vdots&\ddots&\vdots\\
i_{n\,1}&\dots&i_{n,n}\\
\end{bmatrix}
\begin{bmatrix}
\log\lambda_1\\
\vdots\\
\log\lambda_n
\end{bmatrix}=
\begin{bmatrix}
2\pi ik_1\\
\vdots\\
2\pi ik_n
\end{bmatrix},\]
its real part is a homogeneous linear system whose solution implies that $\log|\lambda_j|=0$ for all $j$ and, from the imaginary part of the system we obtain that the argument of 
each $\lambda_j$ is a
rational factor of $2\pi$.\par
Finally, if $\hat{G}\in\mathrm{\widehat{Diff}}(\C^n,0)$ is formally diagonalizable then, there is a formal change of coordinates such that $g^{-1}\circ \hat{G}\circ
g(x)=\mathrm{d}\hat{G}(0)x$ and we make the previous analysis over its linear part $G(x)=\mathrm{d}\hat{G}(0)x$ concluding that, has to be a diagonal one with multiplicative resonant
entries.
\end{proof}
\begin{pro}\label{pro2}
Let $\hat{f}\in\hat{\mathcal{O}}_n$ and $\hat{G}\in\mathrm{\widehat{Diff}}(\C^n,0)$ formally linearizable such that $\hat{G}$ leaves $\hat{f}$ invariant. If the linear part of 
$\hat{G}$ in its Jordan form has a block 
\[
  \begin{pmatrix}
     \lambda&1\\
     0&\lambda
   \end{pmatrix}
,\]
i.e, $G(x_1,\dots,x_n)=(\dots,\lambda x_j+x_{j+1},\lambda x_{j+1},\dots)$
then $\lambda^m=1$ for some $m\in \mathbb{Z}^{+}$ and, $\hat{f}$ after a formal change of coordinates, in the variables related to that block, is a formal series in the $mth$ power 
of 
the second variable, 
\[
  \hat{f}(0,\dots,0,x_j,x_{j+1},0,\dots,0)=l(x_{j+1}^m)\text{ for }l\in\hat{\mathcal{O}}_1.  
\]
\end{pro}
Observe that if the block is bigger its upper sub matrix $2\times 2$ is like the previous one, thus the proposition is true in the general case.
\begin{proof}
Is only necessary to consider the two dimensional case. Let 
\begin{align*}
G(x_1,x_2)&=(\lambda x_1+x_2,\lambda x_2)\text{ and }\\ \hat{f}\circ G(x_1,x_2)&=\sum_{|I|\geq 1}a_I(\lambda x_1+x_2)^i(\lambda x_2)^j=\hat{f}(x_1,x_2)=\sum_{|I|\geq
1}a_Ix_1^ix_2^j,\\
\text{then}\ a_{i,j}&=\sum_{k=0}^{j}C_{i+k,k}\lambda^{i+j-k}a_{i+k,j-k},\ \text{where}\ C_{l,m}=\binom {l}{m}.
\end{align*}
If $\lambda^j\neq 1$ for all $j\in\mathbb{N}$ then $a_{i,0}=a_{i,0}\lambda^i$ implies $a_{i,0}=0$ and $a_{i,1}=\lambda^{i+1}a_{i,1}+C_{i+1,1}\lambda^ia_{i+1,0}$ implies $a_{i,1}=0$, 
repeating this we get that $f\equiv 0$. Therefore, $\lambda^i=1$
for some $i$ such that $a_{i,0}\neq 0$, consider first the case $\lambda=1$,
\begin{align*}
    a_{i,0}&=a_{i,0},\\
    a_{i,1}&=a_{i,1}+C_{i+1,1}a_{i+1,0}\leadsto a_{i,0}=0\quad\text{for } i>0,\\ 
    a_{i,2}&=a_{i,2}+C_{i+1,1}a_{i+1,1}\leadsto a_{i,1}=0\quad\text{for } i>0,
\end{align*}
by induction, suppose that $a_{i,j}=0$ for $i>0$ and $j\leq n$ then
\begin{align*}
    a_{i,n+2}&=a_{i,n+2}+C_{i+1,1}a_{i+1,n+1}\leadsto a_{i,n+1}=0\quad\text{for } i>0,
\end{align*}
hence the only remaining terms are $a_{0,j}$ then $f(x_1,x_2)=l(x_2)$ as we wanted. In a similar way, if $\lambda^m=1$ but $\lambda^n\neq 1$ for $0<n<m$ with $m,n\in\mathbb{N}$,  
\begin{align*}
    a_{i,0}&=a_{i,0}\lambda^i,\quad\text{if }m\!\not|\ i\text{ then }a_{i,0}=0,   
\end{align*}
the next term is \qquad$a_{i,1}=a_{i,1}\lambda^{i+1}+C_{i+1,1}a_{i+1,0}\lambda^i$,\smallskip\\
if $m\ |\ i+1$ we have that $a_{i+1,0}=0$ and together with the previous step $a_{i,0}=0$ for all $i$. If $m\ \not|\ i+1$ we have that $a_{i,1}=0$, using the next term 
\[
  a_{i,2}=a_{i,2}\lambda^{i+2}+C_{i+1,1}a_{i+1,1}\lambda^{i+1},
\]
we can repeat the analysis. If  $m\ |\ i+2$ we have that $a_{i+1,1}=0$ and using the previous step $a_{i,1}=0$ for all $i$. We proceed by an induction argument, suppose that 
$a_{i,j}$ for $j\leq n$ and $i>0$ then 
\begin{align*}
    a_{i,n+1}&=a_{i,n+1}\lambda^{i+n+1}, \quad\text{if }m\!\not|\ (i+n+1) \text{ then } a_{i,n+1}=0,
\end{align*}
as above consider the next term 
\begin{align*}
    a_{i,n+2}&=a_{i,n+2}\lambda^{i+n+2}+C_{r+1,1}a_{i+1,n+1}\lambda^{i+n+1},
\end{align*}
if $m\ |\ (i+n+2)$ then $a_{i+1,n+1}=0$ and using the previous step (where we show that if $m$ does not divide the sum of the sub-indices of $a_{i,n+1}$ then $a_{i,n+1}=0$ ), 
we have $a_{i,n+1}=0$ for all $i$. Finally, for the case $i=0$ note that $a_{0,j}=a_{0,j}\lambda^{j}$ and we can not argue like above, therefore 
$\hat{f}(x_1,x_2)=\hat{l}(x_2^m)$ for 
$l\in\hat{\mathcal{O}}_1$.\par
The higher dimensional case works in the same way, because some part of $\hat{G}$ will be of the form $(\dots,\lambda x_j+x_{j+1},\dots,\lambda x_{j+k-1}+x_{j+k},\lambda 
x_{j+k},\dots)$, 
for a eigenvalue $\lambda$, and making all $x_i=0$ except for $x_{j+k-1}$ and $x_{j+k}$ we can apply the same analysis. Then, 
$\hat{f}(0,\dots,0,x_{j+k-1},x_{j+k},0,\dots,0)=\hat{l}(x_{j+k}^m)\text{ for }\hat{l}\in\hat{\mathcal{O}}_1$ .\par
Finally, if $\hat{G}\in\mathrm{\widehat{Diff}}(\C^n,0)$ is formally linearizable then, there is a formal change of coordinates such that $\hat{g}^{-1}\circ \hat{G}\circ
\hat{g}(x)=\mathrm{d}\hat{G}(0)x$ and we make the previous analysis over its linear part $G(x)=\mathrm{d}\hat{G}(0)x$.
\end{proof}
\begin{defi}
Let $\hat{f}_1\dots,\hat{f}_n\in \hat{\mathcal{O}}_n$. 
\begin{itemize}
 \item We say that $\hat{f}_1\dots,\hat{f}_n$ are \emph{generically transverse} if $\mathrm{d}\hat{f}_1\wedge\dots\wedge\mathrm{d}\hat{f}_n\not\equiv 0$.
 \item We say that $\hat{f}_1\dots,\hat{f}_n$ are \emph{transversally at the origin} if $(\mathrm{d}\hat{f}_1\wedge\dots\wedge\mathrm{d}\hat{f}_n)_0\neq 0$.
\end{itemize}
\end{defi}
\begin{rem}\label{rem_diago}
Observe that in dimension $2$ there can not exist $\hat{f}_1$ and $\hat{f}_2$ transversally independent such that $\hat{f}_i\circ \hat{G}=\hat{f}_i$ with $\hat{G}$ as in the 
proposition above. In a similar way for dimension $n$, there can not exist $\hat{f}_1,\dots,\hat{f}_n$ transversally at the origin such that $\hat{f}_i\circ \hat{G}=\hat{f}_i$ 
with $\hat{G}$ as in the proposition above, because each one satisfies 
\[
  \hat{f}_i(0,\dots,0,x_j,x_{j+1},0,\dots,0)=l_i(x_{j+1}^m)\text{ for some }l_i\in\hat{\mathcal{O}}_1,
\]
and then $\mathrm{d}\hat{f}_1\wedge\dots\wedge\mathrm{d}\hat{f}_n$ is $0$ restricted to the plane $\{x_j,x_{j+1}\}$, in particular 
$(\mathrm{d}\hat{f}_1\wedge\dots\wedge\mathrm{d}\hat{f}_n)_0= 0$.
\end{rem}
Now, using the propositions above we obtain another property, but in this occasion for a group of a formal diffeomorphism leaving invariant a set of generically transverse formal 
series.
\begin{defi}
For $\hat{f}\in\hat{\mathcal{O}}_n$, the \emph{invariance group of $\hat{f}$} is defined as
\[
    H(\hat{f})=\{\hat{G}\in\widehat{\mathrm{Diff}}(\C^n,0)\,|\,\hat{f}\circ\hat{G}=\hat{f}\},
\]
 and the \emph{invariance group of $\{\hat{f}_1\dots,\hat{f}_n\}$},
 \[
    H(\hat{f}_1\dots,\hat{f}_n)=\{\hat{G}\in\widehat{\mathrm{Diff}}(\C^n,0)\,|\,\hat{f}_i\circ\hat{G}=\hat{f}_i\quad\text{for}\quad i=1,\dots,n\}.  
 \]
\end{defi}
The following proposition together with the previous part is one of the key parts of our work.
\begin{pro}\label{ribón}
Let $\hat{f}_1\dots,\hat{f}_n\in \hat{\mathcal{O}}_n$ be generically transverse. Then the group $H(\hat{f}_1\dots,\hat{f}_n)$ is periodic (in 
particular linearizable and finite).
\end{pro}
The demonstration of Proposition \ref{ribón} requires algebraic properties of groups of diffeomorphisms, in Appendix \ref{demRib} we give part of the supporting material and a sketch 
of the proof. Using the theory we have built so far, we can give a proof of the following particular case,
\begin{pro}\label{ribónpart}
Let $\hat{f}_1\dots,\hat{f}_n\in \hat{\mathcal{O}}_n$ be transverse at the origin. Then the group 
$H(\hat{f}_1\dots,\hat{f}_n)$ is periodic (in particular linearizable and finite).
\end{pro}
For the proof of Proposition \ref{ribónpart} we need the following result from \cite{Brochero}, whose demonstration we put here to emphasize that is also valid in the formal case:
\begin{pro}\label{linea-Bro}
A group $\mathcal{G}\subset \mathrm{\widehat{Diff}}(\C^n,0)$ is linearizable if and only if there exists a vector field $\X = \mathcal{R} +\cdots$, where $\mathcal{R}$ is a radial 
vector field, such that $\X$ is invariant for every $\hat{G}\in \mathcal{G}$, i.e. $\hat{G}^*\X =\X$.
\end{pro}
\begin{proof}
\item[$(\Longrightarrow)$] Suppose that $\mathcal{G}$ is linearizable, i.e. there exists $g:(\C^n , 0) \to (\C^n , 0)$
such that $g\circ \mathcal{G}\circ g^{-1} =\{\mathrm{d}\hat{G}_0\,|\hat{G} \in \mathcal{G}\}.$ Since $(A(\cdot) )^*\mathcal{R} = \mathcal{R}$  for all $A\in Gl(n, \C)$ (by a direct 
calculation $(A(\cdot) )^*\mathcal{R}_z =\mathrm{d}A(\cdot)_{A^{-1}z}\mathcal{R}A^{-1}z= z$), in particular for every element $\hat{G}\in \mathcal{G}$ we have
\begin{align*}
\mathcal{R}_z=&\,(g\circ \hat{G}\circ g^{-1})^*\mathcal{R}_z=\mathrm{d}(g\circ \hat{G}\circ g^{-1})_{(g\circ \hat{G}^{-1}\circ g^{-1})(z)}\mathcal{R}((g\circ \hat{G}^{-1}\circ 
g^{-1})(z)),\\
z=&\,\mathrm{d}g_{g^{-1}(z)}\mathrm{d}\hat{G}_{\hat{G}^{-1}\circ g^{-1}(z)}\mathrm{d}g^{-1}_{(g\circ \hat{G}^{-1}\circ g^{-1})(z)}(g\circ \hat{G}^{-1}\circ g^{-1})(z),\\
\text{ taking }&z=g(y)\text{ and multiplying by }\mathrm{d}g^{-1}_{g(y)}\text{ we have, }\\
&\mathrm{d}g^{-1}_{g(y)}(g(y))=\,\mathrm{d}\hat{G}_{\hat{G}^{-1}(y)}\mathrm{d}g^{-1}_{(g\circ \hat{G}^{-1}(y))}(g\circ \hat{G}^{-1}(y)),
\end{align*}
denoting $\X=\mathrm{d}g^{-1}_{g(\cdot)}(g(\cdot))$ we have $\hat{G}^*\X=\X$. It is easy to see that $\X=\mathcal{R}+\cdots$. For this, suppose that
\begin{align*}
 g(z)&=Az+P_l(z)+P_{l+1}(z)+\cdots,\\
 g^{-1}(z)&=A^{-1}z+Q_\nu(z)+Q_{\nu+1}(z)+\cdots,
 \end{align*}
  where $A\in \mathcal{M}_n(\C)$ and $P_l,\,Q_\nu$ are polynomial vector fields of degree $l$ and $\nu$, then
 \begin{align*}
 \mathrm{d}g^{-1}_z&=A^{-1}+\mathrm{d}Q_\nu(z)+\mathrm{d}Q_{\nu+1}(z)+\cdots,\\
 \mathrm{d}g^{-1}_{g(z)}&=A^{-1}+\mathrm{d}Q_\nu(z)_{g(z)}+\mathrm{d}Q_{\nu+1}(z)_{g(z)}+\cdots,
 \end{align*}
 \begin{align*}
 \mathrm{d}g^{-1}_{g(z)}g(z)&=\big(A^{-1}+\mathrm{d}Q_\nu(z)_{g(z)}+\cdots\big)\big(Az+P_l(z)+\cdots \big)\\
 \X_z&=z+A^{-1}\big(P_l(z)+P_{l+1}(z)+\cdots\big)+\\
 &\qquad+\mathrm{d}Q_\nu(z)_{g(z)}\big(Az+P_l(z)+P_{l+1}(z)+\cdots \big)+\cdots
\end{align*}
The terms after $z$, if not $0$, are of degree greater than one. Thus, $\X=\mathcal{R}+\cdots$ as we wanted.
\item[$(\Longleftarrow)$] Since every eigenvalue of the linear part of $\X$ is $1$, then $\X$ is in the Poincaré
domain without resonances (additive resonances), therefore there exists a formal diffeomorphism (using Poincaré linearization theorem, \cite{Ilya_Yako} Theorem 4.3) $g:(\C^n , 0) \to 
(\C^ n , 0)$ such that $g^*\X = \mathcal{R}$, i.e. $\X = (\mathrm{d}g(\cdot) )^{-1}g(\cdot)$.\par
We claim that $g\circ \hat{G}\circ g^{-1} (y) = \mathrm{d}\hat{G}_0(y)$ for every $\hat{G}\in\mathcal{G}$. In fact, from the same
procedure as before we can observe that \[\mathcal{R}_z=\,(g\circ \hat{G}\circ g^{-1})^*\mathcal{R}_z.\]
For this note that $\hat{G}^*\X =\X\,\leadsto$
\begin{align*}
&\mathrm{d}\hat{G}_{\hat{G}^{-1}(y)}\X_{\hat{G}^{-1}(y)}=\X_z\\
&\mathrm{d}\hat{G}_{\hat{G}^{-1}(y)}\mathrm{d}g^{-1}_{g\circ \hat{G}^{-1}(y)}g\circ \hat{G}^{-1}(y)=\mathrm{d}g^{-1}_{g(y)}g(y)\\
&\text{taking }z=g(y), \text{we have}\\
&\mathrm{d}\hat{G}_{\hat{G}^{-1}\circ g^{-1}(z)}\mathrm{d}g^{-1}_{g\circ \hat{G}^{-1}\circ g^{-1}(z)}g\circ \hat{G}^{-1}\circ g^{-1}(z)=\mathrm{d}g^{-1}_z(z)
\end{align*}
Therefore,
\begin{align*}
(g\circ \hat{G}\circ g^{-1})^*\mathcal{R}_z&=\mathrm{d}(g\circ \hat{G}\circ g^{-1})_{g\circ \hat{G}^{-1}\circ g^{-1}(z)}\mathcal{R}(g\circ \hat{G}^{-1}\circ g^{-1}(z))\\
&=\mathrm{d}g_{g^{-1}(z)}\mathrm{d}\hat{G}_{\hat{G}^{-1}\circ g^{-1}(z)}\mathrm{d}g^{-1}_{g\circ \hat{G}^{-1}\circ g^{-1}(z)}g\circ \hat{G}^{-1}\circ g^{-1}(z)\\
&=\mathrm{d}g_{g^{-1}(z)}\mathrm{d}g^{-1}_z(z)\quad\text{(by the the previous computation)}\\
&=z.
\end{align*}
Now, if we suppose that $g\circ \hat{G}\circ g^{-1}(z) = Az + P_l(z) + P_{l+1}(z) +\cdots$, where $P_j(z)$ is
a polynomial vector field of degree $j$, then it is easy to prove that
\[(g \circ \hat{G} \circ g^{-1})^*\mathcal{R} = Az + lP_l (z) + (l + 1)P_{l+1} (z) +\cdots,\]
In order to prove it, observe that $(g\circ \hat{G}\circ g^{-1})^*\mathcal{R}_z=\mathcal{R}_z\,\leadsto$
\[\mathrm{d}(g\circ \hat{G}\circ g^{-1})_{g\circ \hat{G}^{-1}\circ g^{-1}(y)}\mathcal{R}(g\circ \hat{G}^{-1}\circ g^{-1}(y))=\mathcal{R}(y)\]
taking  $y=g\circ \hat{G}\circ g^{-1}(z)$ then
\begin{align*}
\mathrm{d}(g\circ \hat{G}\circ g^{-1})_z\mathcal{R}(z)&=\mathcal{R}(g\circ \hat{G}\circ g^{-1}(z)),\\
\mathrm{d}(g\circ \hat{G}\circ g^{-1})_zz&=g\circ \hat{G}\circ g^{-1}(z),\\
\text{by hypothesis}\quad\mathrm{d}(g\circ \hat{G}\circ g^{-1})_z&=A + \mathrm{d}(P_l)_z + \mathrm{d}(P_{l+1})_z +\cdots,\\
\mathrm{d}(g\circ \hat{G}\circ g^{-1})_zz&=Az + lP_l(z) + (l+1)P_{l+1}(z) +\cdots,\\
&=Az + P_l(z) + P_{l+1}(z) +\cdots.
\end{align*}
and therefore $P_j (z)\equiv 0$ for every $j\geq 2$.
\end{proof}
\begin{proof}[Proof of Proposition \ref{ribónpart}]
The idea is to use the above proposition, so that we need to find an invariant vector field $\X$. First, consider the formal map $H=(\hat{f}_1,\dots,\hat{f}_n)$, for each 
$\hat{G}\in\mathcal{G}$ by hypothesis $\hat{f}_i\circ \hat{G}=f_i$ then we have $H\circ \hat{G}=H$, and note that $H\in\mathrm{\widehat{Diff}}(\C^n,0)$ because 
$(\mathrm{d}\hat{f}_1\wedge\dots\wedge\mathrm{d}\hat{f}_n)_0\neq 0$. Thus, this implies $H\circ \hat{G}^{-1}=H$, $\hat{G}\circ H^{-1} =H^{-1}$ and 
$\mathrm{d}\hat{G}_{H^{-1}(\cdot)}\mathrm{d}H^{-1}_{(\cdot)}=\mathrm{d}H^{-1}_{(\cdot)}$.\smallskip\par Therefore define $\X=(\mathrm{d}H)^{-1}H=\mathrm{d}H^{-1}_{H(\cdot)}H(\cdot)$ 
which satisfies $\hat{G}^*\X =\X$,
\begin{align*}
\hat{G}^*\X_z&=\mathrm{d}\hat{G}_{\hat{G}^{-1}(z)}\X_{\hat{G}^{-1}(z)},\\
&=\mathrm{d}\hat{G}_{\hat{G}^{-1}(z)}\mathrm{d}H^{-1}_{H(\hat{G}^{-1}(z))}H(G^{-1}(z)),\\
&=\mathrm{d}\hat{G}_{\hat{G}^{-1}(z)}\mathrm{d}H^{-1}_{H(z)}H(z),\\
&=\mathrm{d}H^{-1}_{H(z)}H(z),\ \text{because }\mathrm{d}\hat{G}_{H^{-1}(H\circ \hat{G}^{-1}(z))}\mathrm{d}H^{-1}_{H\circ \hat{G}^{-1}(z)}=\mathrm{d}H^{-1}_{H\circ G^{-1}(z)},\\
\hat{G}^*\X_z&=\X_z.
\end{align*}
And, as in the proof of Proposition \ref{linea-Bro} we have that $\X=\mathcal{R}+\cdots$.\par
Then by the Proposition \ref{linea-Bro} we have that $\mathcal{G}$ is linearizable. Furthermore, this implies that $\mathcal{G}$ is in fact diagonalizable by Propositions 
\ref{pro1} and \ref{pro2}, and Remark \ref{rem_diago}, furthermore its diagonal form is made of roots of the unity because the transversally condition of 
$\{\hat{f}_i\}$ implies the existence of $n$ independent multi-indexes, which is the next step in the proof.\par 
Working for simplicity in dimension two, write $\hat{f}_1(x)=\sum_Ia_Ix^I$ and $\hat{f}_2(x)=\sum_Jb_Jx^J$ where $x=(x_1,x_2),\ I=(i,j),$ and $J=(r,s)$ then
\[\mathrm{d}\hat{f}_1\wedge\mathrm{d}\hat{f}_2=\big(\sum_{I,J}a_Ib_J(is-jr)x_1^{i+r-1}x_2^{j+s-1}\big)\mathrm{d}x_1\wedge\mathrm{d}x_2,\]
if there were no $I,J$ independent such that $a_Ib_J\neq 0$ we would have $\mathrm{d}\hat{f}_1\wedge\mathrm{d}\hat{f}_2\equiv 0$ contradicting the hypothesis, so there exists a 
couple $I_0=(i_0,j_0)$, $J_0=(r_0,s_0)$ with this condition. Consider $\hat{G}\in\mathcal{G}$ and $G=(\mathrm{d}\hat{G})_0$ its linear part given by a diagonal matrix with 
eigenvalues 
$\lambda_1,\lambda_2$, the conditions $\hat{f}_i\circ G=\hat{f}_i$ for $i=1,2$ implies $\lambda_1^{i_0}\lambda_2^{j_0}=1$ and $\lambda_1^{r_0}\lambda_2^{s_0}=1$ respectively and, as 
before this implies that both are roots of the unity. Indeed, the previous analysis is more subtle, because we have to consider $(\hat{f}_i\circ \hat{G})(g)=(\hat{f}_i\circ 
g)(g^{-1}\circ\hat{G}\circ g)=(\hat{f}_i\circ g)(G) =(\hat{f}_i\circ g)$ where $g$ is a formal diffeomorphism who diagonalizes $\mathcal{G}$, the result is the same because the 
$\hat{f}_i\circ g$ are generically transverse.\par
In general we have something like 
$\mathrm{d}\hat{f}_1\wedge\dots\wedge\mathrm{d}\hat{f}_n\not\equiv 0$ and $\hat{f}_i(x)=\sum_I{_{i}}a_Ix^I$ with 
$I=(i_1,\dots,i_n)$, but the associativity of the wedge product allow us to work in pairs, for instance 
\[\mathrm{d}\hat{f}_1\wedge\mathrm{d}\hat{f}_2=\sum_{r<s}\Big(\frac{\partial \hat{f}_1}{\partial x_r}\frac{\partial \hat{f}_2}{\partial x_s}-\frac{\partial \hat{f}_1}{\partial 
x_s}\frac{\partial \hat{f}_2}{\partial x_r}\Big)\mathrm{d}x_r\wedge\mathrm{d}x_s,\] each therm of the sum works like the previous case and at list one of them should be not zero 
meaning that it exist a couple $(i_r,i_s),(j_r,j_s)$ independent and with this $I=(i_1,\dots,i_r,\dots,i_s,\dots,i_n)$ and $J=(j_1,\dots,j_r,\dots,j_s,\dots,j_n)$ are independent and 
its coefficients are not zero $_{1}a_I{_{2}}a_J\neq 0$. Thus, the following sum is not zero,
\[_{1}a_I{_{2}}a_J\sum_{r<s}\begin{vmatrix}i_r&i_s\\j_r&j_s\end{vmatrix}x^{I+J-(e_r+e_s)}\mathrm{d}x_r\wedge\mathrm{d}x_s,\] where 
$I+J-(e_r+e_s)=(i_1+j_1,\dots,i_r+j_r-1,\dots,i_s+j_s-1,\dots,i_n+j_n)$. The wedge product with the next form, $\mathrm{d}\big(\hat{f}_3(x)\big)=\mathrm{d}(\sum_K{_{3}}a_Kx^K)$,  
will 
produce terms having $3\times 3$ matrices related to the multi-indexes $I,J$ and $K$, and obviously the dependence of $K$ with $I,J$ would imply that all of them are zero. This 
process continues implying the existence of $n$ independent multi-indexes such that $\lambda_1^{i_1}\cdots\lambda_n^{i_n}=1$ for each one of them, and the $\lambda_i$'s are roots of 
the unity as before.
   \par Therefore, there exists $N\in\mathbb{N}$ such that $G^N=I$ and then $\langle \hat{G}\rangle$ (i.e. the group generated by $\hat{G}$) is finite. It remains to prove that 
$\mathcal{G}$ is commutative, consider $\hat{G}_1,\hat{G}_2\in\mathcal{G}$ and note by $G_1,G_2$ their linear parts, then  
\begin{align*}
	\hat{G}_1\circ\hat{G}_2&=g(g^{-1}\circ\hat{G}_1\circ g)(g^{-1}\circ\hat{G}_2\circ g)g^{-1}\\
	&=g(G_1\circ G_2)g^{-1}\quad\text{they commute},\\
	&=g(G_2\circ G_1)g^{-1},\\
	&=\hat{G}_2\circ\hat{G}_1.
\end{align*}
\end{proof}
\chapter{On formal first integrals}\label{Chap3}
We will show that the existence of a formal first integral in our framework, implies the existence of a holomorphic one. 
\section{Preliminaries}
In this section we are strongly based in the notation and results of \cite{Cam-Sca2,Cam-Sca1} that we write next for the sake of
completeness.\par
\begin{defi}\label{generic}
We shall say that $\F(\X)$ is \emph{non-degenerate generic} if $\mathrm{d}\X(0)$ is non-singular, diagonalizable, and after some suitable change of coordinates $\X$ leaves invariant
the coordinate planes. Denote the set of germs of non-degenerate generic vector fields on $(\C^n , 0)$ by $\mathrm{Gen}\big(\mathfrak{X}(\C^n , 0)\big)$. Such vector fields
 after a change of coordinates can be written in the form
\begin{equation}\label{gen-vf}
  \X(x) = \lambda_1x_1 (1 + a_1(x))\frac{\partial}{\partial x_1}+\dots+\lambda_nx_n (1 + a_n(x))\frac{\partial}{\partial x_n},
\end{equation}
where $a_i\in\mathcal{M}_3$ for $i=1,\dots,n$.
\end{defi}
\begin{defi}\label{firstintegral}
We say that a germ of a holomorphic foliation $\F(\X)$ has a \emph{holomorphic first integral}, if there is a germ of a holomorphic map $F: (\C^n,0)\to (\C^{n-1},0)$ such that:
\begin{enumerate}[(a)]
\item $F$ is a submersion off some proper analytic subset. Equivalently if we write $F = (f_1 ,\dots, f_{n-1})$ in coordinate
functions, then the $(n-1)$-form $\mathrm{d}f_1\wedge\dots\wedge\mathrm{d}f_{n-1}$ is non-identically zero.
\item  The leaves of $\F(\X)$ are contained in level curves of $F$ .
\end{enumerate}
Further, a germ $f$ of a meromorphic function at the origin $0\in\C^n$ is called $\F(\X)$-\emph{invariant} if
the leaves of $\F(\X)$ are contained in the level sets of $f$. This can be precisely stated in terms of
representatives for $\F(\X)$ and $f$, but can also be written as $i_{\X}(\mathrm{d}f)=\X(f)\equiv 0$.\par
\end{defi}
We start with the following definition inspired by the definition of \emph{holomorphic fist integral} (Definition \ref{firstintegral}), though it will not be used until the end of 
the 
article is necessary to settle down the framework we use.
\begin{defi}[formal first integral]
We say that a germ of a holomorphic foliation $\F(\X)$, were $\X\in\mathfrak{X}(\C^n,0)$, has a \emph{formal first integral}, if there is a formal map 
$\hat{F}=(\hat{f}_1,\dots,\hat{f}_{n-1})$, with $\hat{f}_1,\dots,\hat{f}_{n-1}\in\hat{\mathcal{O}}_n$, such that:
\begin{enumerate}[(a)]
\item The formal $(n-1)$-form $\mathrm{d}\hat{f}_1\wedge\dots\wedge\mathrm{d}\hat{f}_{n-1}$ is non-identicaly zero.
\item $\X(\hat{F})\equiv 0$, (i.e. $\X(\hat{f}_i)\equiv 0$ for all $\hat{f}_i,\ i=1,\dots,n-1$ ).
\end{enumerate}
\begin{defi}[condition ($\star$)]\label{star}
Let $\X$ be a germ of a holomorphic vector field at the origin such that the origin $0\in\C^m , m\ge 3$ is a nondegenerate singularity of $\X$ (i.e. $\mathrm{d}\X(0)$ is 
non-singular). We say that $\X$ satisfies \emph{condition} $(\star)$ if there is a real line $L\subset \C$ through the origin, separating a certain eigenvalue $\lambda(\X)$ from the 
others.
If $\X$ satisfies ($\star$) we denote by $S_{\X}$ the smooth invariant curve associated to $\lambda(\X)$.
\end{defi}
\end{defi}
Though the methods we use in this chapter are, in general, independent of the dimension, our work will imply directly the condition ($\star$) only when $n=3$, in the remaining cases 
we have to include it as a hypothesis. This condition, together with the generic conditions of the vector field $\X$, is what allows to use the following well known result 
\cite{Eli-Il} whose 
demonstration can also be found in \cite{Reis}.
\begin{theorem}\label{rusos}
 Let $\X$ and $\mathcal{Y}$ be two vector fields in $\mathrm{Gen}\left(\mathfrak{X}(\C^n,0)\right)$ with an isolated singularity at the origin and satisfying condition ($\star$). Let 
$h_{\X}$ and $h_{\mathcal{Y}}$ be the holonomies of $\X$ and $\mathcal{Y}$ relatively to $S_{\X}$ and $S_{\mathcal{Y}}$, respectively. Then $\X$ and $\mathcal{Y}$ are analytically 
equivalent if and only if the holonomies $h_{\X}$ and $h_{\mathcal{Y}}$ are analytically conjugate.
\end{theorem}
This theorem is basically the heart of the proof of $(3)\Leftrightarrow(4)$ in Theorem 1 of \cite{Cam-Sca1} whose statement is:
\begin{theorem}\label{existence}
Suppose that $\X\in\mathrm{Gen}(\mathfrak{X}(\C^3 , 0))$ satisfies condition $(\star)$ and let $S_{\X}$ be the axis associated to the separable eigenvalue of $\X$.\par Then, 
$\mathrm{Hol}(\F(\X), S_{\X}, \Sigma)$ is periodic (in particular linearizable and finite) if and only if $\F(\X)$ has a holomorphic first integral.
\end{theorem}
In order to prove our result we show that having a formal first integral, gives enough properties to the vector field that Theorem \ref{existence} can be used.
\section{Algebraic criterion}
In this section we show that we can restrict ourselves to a vector fields written in a particular way.\par The following lemma and proposition are, at first glance, mostly $n$ 
dimensional versions of Lemma 1 and Proposition 1 in \cite{Cam-Sca1}. Nevertheless, there is a big difference
which turns out to be an important property as we will see later.
\begin{lem}\label{lem-alg}
Let $\Lambda=(\lambda_1,\dots,\lambda_n)\in\C^n\setminus 0$ and, let $N_{n-1\times n}$ be a matrix with entries in $\mathbb{N}$ and linearly independent lines, satisfying
\[N\Lambda^t=0\in\C^{n-1}.\] Then there are $k_1\dots,k_{n}\in\mathbb{Z}$ and $\lambda\in\C^*$ such that
\[(\lambda_1,\dots,\lambda_n)=(k_1\dots,k_{n})\lambda.\]
\end{lem}
\begin{proof}
The proof consists in the solution of a linear system, take
\[N=\begin{bmatrix}
n_{1\,1}&\dots&n_{1\,n-1}&n_{1\,n}\\
\vdots&\ddots&\vdots&\vdots\\
n_{n-1\,1}&\dots&n_{n-1\,n-1}&n_{n-1\,n}
\end{bmatrix}\text{ and } A=\begin{bmatrix}
n_{1\,1}&\dots&n_{1\,n-1}\\
\vdots&\ddots&\vdots\\
n_{n-1\,1}&\dots&n_{n-1\,n-1}
\end{bmatrix},\] the independence allows to take $n-1$ independent columns, suppose the first ones, and form the matrix $A$ which is invertible, thus multiplying by $A^{-1}$ the 
system
$N\Lambda^t=0$ we get,
\[\begin{bmatrix}
1&\dots&0&\tilde{k}_1\\
\vdots&\ddots&\vdots&\vdots\\
0&\dots&1&\tilde{k}_{n-1}
\end{bmatrix}
\begin{bmatrix}
\lambda_1\\\vdots\\\lambda_n
\end{bmatrix}=
\begin{bmatrix}
0\\\vdots\\0
\end{bmatrix}_{n-1\times 1}, \]
 and we have $n-1$ equation of the form $\lambda_i+\tilde{k}_i\lambda_n=0$, then
 \[(\lambda_1,\dots,\lambda_n)=(-\tilde{k}_1,\dots,-\tilde{k}_{n-1},1)\lambda_n,\]
 we know exactly who are the $\tilde{k}_i$'s, because they satisfy
 \[\begin{bmatrix}
n_{1\,1}&\dots&n_{1\,n-1}\\
\vdots&\ddots&\vdots\\
n_{n-1\,1}&\dots&n_{n-1\,n-1}\end{bmatrix}
\begin{bmatrix}
\tilde{k}_1\\\vdots\\\tilde{k}_{n-1}\end{bmatrix}=
\begin{bmatrix}
n_{1\,n}\\\vdots\\n_{n-1\,n},
\end{bmatrix},\]
and by the Cramer rule, $\tilde{k}_i=\frac{|A_i|}{|A|}$, where $|\cdot|$ means determinant and $A_i$ is the matrix $A$ changing the column $i$ by $[n_{1\,n}\,\dots\,n_{n-1\, n}]^t$.
Finally we get, \[(\lambda_1,\dots,\lambda_n)=(|A_1|,\dots,|A_{n-1}|,-|A|)\lambda,\] with $\lambda=-\lambda_n/|A|$ and $k_i=|A_i|,k_n=-|A|\in\mathbb{Z}$ for $i=1,\dots,n-1$ as we
wanted.
\end{proof}
The three dimensional case is especial because we know that $k_1\cdot k_2\cdot k_3<0$, so we can make one of them negative and the others positive by changing the $\lambda$. However, 
in dimension $n>3$, the only thing we know about the signs of the $k_i$ is that can not be all positive nor negative thanks to the condition $n_{1\,1}k_1+\dots+n_{1\,n}k_{n}=0$. Here 
we have an example in dimension $4$ where $k_1\cdot k_2\cdot k_3\cdot k_3>0$, take
\[N=\begin{bmatrix}
1&0&1&0\\
0&1&1&2\\
0&0&1&1\\
\end{bmatrix},\]
if satisfies $N\Lambda^t=0$ for some  $\Lambda=(\lambda_1,\lambda_2,\lambda_3,\lambda_4)$ then,
\[(\lambda_1,\lambda_2,\lambda_3,\lambda_4)=(-1,1,1-1)\lambda.\]
With this example we can also see that a vector field of Siegel type not necessarily satisfies condition ($\star$) while the contrary is always true.
\begin{pro}\label{pro-alg}
Suppose that $\X\in\mathrm{Gen}(\mathfrak{X}(\C^n , 0))$ has a formal first integral, then $\F(\X)$ can be given in local coordinates by a vector field of the form 
\begin{equation}\label{formalform}
\X(x) = k_1x_1 (1+ a_1(x))\frac{\partial}{\partial x_1}+\dots+k_nx_n (1 + a_n(x))\frac{\partial}{\partial x_n} 
\end{equation}
 where $k_1,\dots k_n\in\mathbb{Z}$ and $a_1,\dots,a_n\in\mathcal{M}_n$. In particular if $n=3$, $\X$ satisfies condition $(\star)$.
\end{pro}
\begin{proof}
We are considering $\X\in\mathrm{Gen}(\mathfrak{X}(\C^n , 0))$, and by definition, suppose now that
$\hat{F}=(\hat{f}_1\dots,\hat{f}_{n-1})$ is the formal first integral, then $\X(\hat{f}_i)\equiv 0$ for $i=1,\dots,n-1$. If $\hat{f}_i(x)=\sum_{|I|>p_i}{_{i}}a_Ix^I$
then \[\frac{\partial \hat{f}_i}{\partial x_r}(x)=\sum_{|I|>p_i}(i_r){_{i}}a_Ix_1^{i_1}\cdots x_r^{i_r-1}\cdots x_n^{i_n},\] and
\begin{align*}
\X(\hat{f}_i)&=\sum_{r=1}^n\lambda_rx_r (1 + a_r(x))\big(\sum_{|I|>p_i}(i_r){_{i}}a_Ix_1^{i_1}\cdots x_r^{i_r-1}\cdots x_n^{i_n}\big),\\
&=\sum_{r=1}^n\sum_{|I|>p_i}i_r\lambda_r{_{i}}a_I(1 + a_r(x))x_1^{i_1}\cdots x_r^{i_r}\cdots x_n^{i_n},\\
&=\sum_{|I|>p_i}\sum_{r=1}^ni_r\lambda_r{_{i}}a_I(1 + a_r(x))x^I,\\
&=\sum_{|I|>p_i}a_I\Big(\sum_{r=1}^ni_r\lambda_r{_{i}}\Big)x^I + \sum_{|I|>p_i}a_I\Big(\sum_{r=1}^ni_r\lambda_r{_{i}}a_r(x)\Big)x^I,\\
J^{p_i}\X(\hat{f}_i)&=\sum_{|I|=p_i}a_I\Big(\sum_{r=1}^ni_r\lambda_r{_{i}}\Big)x^I=0,
\end{align*}
then $\sum_{r=1}^ni_r\lambda_r{_{i}}=0$ for each $I=(i_1,\dots,i_n)$ when $a_I\neq 0$. Now, as we show at the end of the proof of Proposition \ref{ribónpart} there are $n-1$ 
linearly independent $n$-tuples satisfying this condition and with them we can form the matrix $N$ of Lemma \ref{lem-alg}, and we are done.\par
\end{proof}
\section{Holonomy and formal first integrals}
We know that holonomy maps (by its construction) leave invariant the level sets of a holomorphic first integral. What we want to obtain is a similar invariant relation in the case of 
a formal one, for
simplicity we work in dimension $3$ but small changes are needed for the general case. Consider the foliation given by
\[\X(x_1,x_2,x_3)=px_1a_1(x)\frac{\partial}{\partial x_1}+qx_2a_2(x)\frac{\partial}{\partial x_2}+x_3\frac{\partial}{\partial x_3},\]
where $a_1,a_2\in \mathcal{M}_3$ and $p,q\in\mathbb{Q}$, be $S:= (x_1 = x_2 = 0)$ and $\Sigma := (x 3 = 1)$. Now consider the closed loop $\gamma :[0,1]\mapsto S$ given by $\gamma(t)
= (0,0, e^{2\pi it})$ and let $\overline{\Gamma}_{(x_1 ,x_2 )}(t) = (\Gamma_1 (x_1, x_2, t), \Gamma_2 (x_1, x_2, t), e^{2\pi it})$ be its lifting along the leaves of $\F(\X)$ starting
at $(x_1, x_2, 1)\in\Sigma$. In particular, the map $h\in\mathrm{Diff}(\C^2 , 0)$ given by $\overline{\Gamma}_{(x_1 ,x_2)}(1) = (h(x_1, x_2 ), 1)$ is a generator of
$\mathrm{Hol}(\F(\X), S, \Sigma)$. 
Since $\overline{\Gamma}_{(x_1 ,x_2 )}(t)$ belongs to a leaf of $\F(\X)$, then 
\[\frac{\partial}{\partial t}\overline{\Gamma}_{(x_1,x_2)}(t)=\alpha\X(\Gamma_1 (x_1, x_2, t), \Gamma_2 (x_1 , x_2, t ), e^{2\pi it}).\]
From this vector equation one has $2\pi i e^{2\pi it}=\alpha e^{2\pi it}$, thus $\alpha = 2\pi i$. Furthermore, 
\begin{align*}
\frac{\partial \Gamma_1}{\partial t}  &=2p\pi i\Gamma_1(x_1, x_2, t)a_1(\overline{\Gamma}),\\
\frac{\partial \Gamma_2}{\partial t}  &=2q\pi i\Gamma_2(x_1, x_2, t)a_2(\overline{\Gamma}).
\end{align*}
\begin{rem}
Note that by Proposition \ref{pro-alg}, we can take the vector field $\X$ in the form \eqref{formalform} and multiplying by $\big(-k_3(1+a_3(x))\big)^{-1}$ is obtained a the vector 
field 
like the one we are using in this section who defines the same foliation.
\end{rem}
Suppose that $\hat{F}=(\hat{f}_1,\hat{f}_2)$, with $\hat{f}_1,\hat{f}_2\in \hat{\mathcal{O}}_3$, is a  formal first integral of the foliation $\F(\X)$, this means 
that $\hat{f}_1$ and $\hat{f}_2$ are $\F(\X)$-invariant then,
\begin{align*}
0&=px_1a_1(x_1,x_2,x_3)\frac{\partial\hat{f}_1}{\partial x_1}+qx_2a_2(x_1,x_2,x_3)\frac{\partial\hat{f}_1}{\partial x_2}+x_3\frac{\partial\hat{f}_1}{\partial x_3},\\
&\quad\text{evaluating }\overline{\Gamma}\text{ and multiplying by }2\pi i,\\
0&=2\pi ip\Gamma_1a_1(\overline{\Gamma})\frac{\partial\hat{f}_1}{\partial x_1}\Big|_{\overline{\Gamma}}+2\pi iq\Gamma_2a_2(\overline{\Gamma})\frac{\partial\hat{f}_1}{\partial
x_2}\Big|_{\overline{\Gamma}}+2\pi ie^{2\pi it}\frac{\partial\hat{f}_1}{\partial x_3}\Big|_{\overline{\Gamma}},\\
0&=\frac{\partial\Gamma_1}{\partial t}\frac{\partial\hat{f}_1}{\partial x_1}\Big|_{\overline{\Gamma}}+\frac{\partial\Gamma_2}{\partial t}\frac{\partial\hat{f}_1}{\partial
x_2}\Big|_{\overline{\Gamma}}+\frac{\mathrm{d}}{\mathrm{d}t}(e^{2\pi it})\frac{\partial\hat{f}_1}{\partial x_3}\Big|_{\overline{\Gamma}},\\
0&=\frac{\partial}{\partial t}(\hat{f}_1\circ\overline{\Gamma}).
\end{align*}
The last line (which also has for $\hat{f}_2$) implies that $\hat{f}_1\circ\overline{\Gamma}$ is constant in $t$, then,
\begin{align*}
\hat{f}_1\circ\overline{\Gamma}(x_1,x_2,1)&=\hat{f}_1\circ\overline{\Gamma}(x_1,x_2,0),\\
\hat{f}_1(h(x_1,x_2),1)&=\hat{f}_1(x_1,x_2,1).
\end{align*}
In conclusion, we obtain the relation we were looking for: \[\hat{F}(h(x_1,x_2),1)=\hat{F}(x_1,x_2,1).\]
\begin{rem}\label{hol_rad}
  Note that the previous computation works in the same way if we use instead of $\gamma(t)$ a circle with small radius. Note also that we are using the formal chain rule Section 
\ref{chainrule}.
\end{rem}
\subsection{From formal to holomorphic first integral}
Now we are in conditions to prove our first main result:
\begin{theorem}\label{theoA}
Let $\F(\X)$ be the germ of a holomorphic foliation with $\X\in\mathrm{Gen}\left(\mathfrak{X}(\C^3,0)\right)$, if $\F(\X)$ possesses a formal first integral   then it also possesses a
holomorphic one.
\end{theorem}
\begin{proof}[Proof of Theorem \ref{theoA}]
By definition of formal first integral  $\mathrm{d}\hat{f}_1\wedge\mathrm{d}\hat{f_2}\neq 0$ and by Proposition 1 in \cite{Cam-Sca2}, the vector field $\X$ can be written as:
\[\X(x)=mx_1(1+a_1(x))\frac{\partial}{\partial x_1}+nx_2(1+a_2(x))\frac{\partial}{\partial x_2}-kx_3(1+a_3(x))\frac{\partial}{\partial x_3},\]
were $m,n,k\in\mathbb{Z}^{+}$ and $a_1,a_2,a_3\in\mathcal{M}_3$ in particular satisfies \emph{condition} $(\star)$.\par Observe that $\X(f_i)\equiv 0$, where $f_i(x)=\sum_I 
{}_ia_Ix^I$, written in the particular case where $x_1=x_2=0$ becomes 
\[
  x_3\big(1+a_3(0,0,x_3)\big)\sum_ka_{0,0,k}kx_3^{k-1}\equiv 0,
\]
we can suppose that $1+a_3(0,0,x_3)\not\equiv 0$ because the vector field has an isolated singularity at the origin. Therefore, $\sum_ka_{0,0,k}kx_3^k\equiv 0$ what implies that 
$a_{0,0,k}\equiv 0$. Using this, we can define formal series in two variables as $\tilde{f}(x_1,x_2):=\hat{f}_i(x_1,x_2,1)$, thus the equalities from the end of the previous 
section become $\tilde{f}_i(x_1,x_2)=\tilde{f}_i(h(x_1,x_2))$ for $i=1,2$.\\
We can use now the previous sections and Chapter \ref{Chap2}, but first, we have to guarantee that they are still generically transverse because in general 
$\mathrm{d}\hat{f}_1\wedge\mathrm{d}\hat{f}_2\not\equiv 0$ does not imply $\mathrm{d}(\hat{f}_1(x_1,x_2,1))\wedge\mathrm{d}(\hat{f}_2(x_1,x_2,1))\not\equiv 0$. If 
$\mathrm{d}\tilde{f}_1\wedge\mathrm{d}\tilde{f}_2\equiv 0$ then, for  $\overline{\Gamma}_{(x_1 ,x_2 )}(t) = (\Gamma_1 (x_1, x_2, t), \Gamma_2 (x_1, x_2, t), e^{2\pi 
it})$ as before, $\mathrm{d}(\hat{f}_1\circ \tilde{\Gamma})\wedge\mathrm{d}(\hat{f}_2\circ \tilde{\Gamma})\equiv 0$ because from the 
previous section we have that $\frac{\partial}{\partial t}(\hat{f}_i\circ\overline{\Gamma})=0$. Let us write this with more care,
\begin{equation}\label{forma}
 \mathrm{d}\hat{f}_1\wedge\mathrm{d}\hat{f}_2=\sum_{i<j}\Big(\frac{\partial\hat{f}_1}{\partial x_i}\frac{\partial\hat{f}_2}{\partial x_j}-\frac{\partial\hat{f}_1}{\partial 
x_j}\frac{\partial\hat{f}_2}{\partial x_i}\Big)\mathrm{d}x_i\wedge\mathrm{d}x_j,
\end{equation}
then $\mathrm{d}\tilde{f}_1\wedge\mathrm{d}\tilde{f}_2=\mathrm{d}(\hat{f}_1(x_1,x_2,1))\wedge\mathrm{d}(\hat{f}_2(x_1,x_2,1))$ is the first them of the sum \eqref{forma} evaluated in 
$(x_1,x_2,1)$. Now,
\[
 \mathrm{d}(\hat{f}_1\circ\tilde{\Gamma})\wedge\mathrm{d}(\hat{f}_2\circ\tilde{\Gamma})=\Big(\frac{\partial(\hat{f}_1\circ\tilde{\Gamma})}{\partial x_1}\frac{
\partial(\hat{f}_2\circ\tilde{\Gamma})}{\partial x_2}-\frac{\partial(\hat{f}_1\circ\tilde{\Gamma})}{\partial x_2}\frac{\partial(\hat{f}_2\circ\tilde{\Gamma})}{\partial 
x_1}\Big)\mathrm{d}x_1\wedge\mathrm{d}x_2,
\]
the other two terms in this sum disappear because they involve derivatives with respect to $t$. Taking into account that 
\[
 \frac{\partial}{\partial x_j}(\hat{f}_i\circ\tilde{\Gamma})=\frac{\partial\hat{f}_i}{\partial x_1}\Big|_{\tilde{\Gamma}}\frac{\partial \Gamma_1}{\partial 
x_j}+\frac{\partial\hat{f}_i}{\partial x_2}\Big|_{\tilde{\Gamma}}\frac{\partial \Gamma_2}{\partial x_j}\quad\text{for}\quad i,j\in\{1,2\}
\]
we have,
\begin{align*}
 \frac{\partial(\hat{f}_1\circ\tilde{\Gamma})}{\partial x_1}\frac{\partial(\hat{f}_2\circ\tilde{\Gamma})}{\partial x_2}&=\Big(\frac{\partial\hat{f}_1}{\partial 
x_1}\Big|_{\tilde{\Gamma}}\frac{\partial \Gamma_1}{\partial 
x_1}+\frac{\partial\hat{f}_1}{\partial x_2}\Big|_{\tilde{\Gamma}}\frac{\partial \Gamma_2}{\partial x_1}\Big)\Big(\frac{\partial\hat{f}_2}{\partial 
x_1}\Big|_{\tilde{\Gamma}}\frac{\partial \Gamma_1}{\partial 
x_2}+\frac{\partial\hat{f}_2}{\partial x_2}\Big|_{\tilde{\Gamma}}\frac{\partial \Gamma_2}{\partial x_2}\Big),\\
&=\frac{\partial\hat{f}_1}{\partial x_1}\Big|_{\tilde{\Gamma}}\frac{\partial\hat{f}_2}{\partial x_1}\Big|_{\tilde{\Gamma}}\frac{\partial 
\Gamma_1}{\partial x_1}\frac{\partial \Gamma_1}{\partial x_2}
+\frac{\partial\hat{f}_1}{\partial x_1}\Big|_{\tilde{\Gamma}}\frac{\partial\hat{f}_2}{\partial x_2}\Big|_{\tilde{\Gamma}}\frac{\partial 
\Gamma_1}{\partial x_1}\frac{\partial \Gamma_2}{\partial x_2}+\\
&\quad +\frac{\partial\hat{f}_1}{\partial x_2}\Big|_{\tilde{\Gamma}}\frac{\partial\hat{f}_2}{\partial x_1}\Big|_{\tilde{\Gamma}}\frac{\partial 
\Gamma_2}{\partial x_1}\frac{\partial \Gamma_1}{\partial x_2}+\frac{\partial\hat{f}_1}{\partial x_2}\Big|_{\tilde{\Gamma}}\frac{\partial\hat{f}_2}{\partial 
x_2}\Big|_{\tilde{\Gamma}}\frac{\partial \Gamma_2}{\partial x_1}\frac{\partial \Gamma_2}{\partial x_2},
\end{align*}
\begin{align*}
 \frac{\partial(\hat{f}_1\circ\tilde{\Gamma})}{\partial x_2}\frac{\partial(\hat{f}_2\circ\tilde{\Gamma})}{\partial x_1}&=\Big(\frac{\partial\hat{f}_1}{\partial 
x_1}\Big|_{\tilde{\Gamma}}\frac{\partial \Gamma_1}{\partial 
x_2}+\frac{\partial\hat{f}_1}{\partial x_2}\Big|_{\tilde{\Gamma}}\frac{\partial \Gamma_2}{\partial x_2}\Big)\Big(\frac{\partial\hat{f}_2}{\partial 
x_1}\Big|_{\tilde{\Gamma}}\frac{\partial \Gamma_1}{\partial 
x_1}+\frac{\partial\hat{f}_2}{\partial x_2}\Big|_{\tilde{\Gamma}}\frac{\partial \Gamma_2}{\partial x_1}\Big),\\
&=\frac{\partial\hat{f}_1}{\partial x_1}\Big|_{\tilde{\Gamma}}\frac{\partial\hat{f}_2}{\partial x_1}\Big|_{\tilde{\Gamma}}\frac{\partial 
\Gamma_1}{\partial x_1}\frac{\partial \Gamma_1}{\partial x_2}
+\frac{\partial\hat{f}_1}{\partial x_1}\Big|_{\tilde{\Gamma}}\frac{\partial\hat{f}_2}{\partial x_2}\Big|_{\tilde{\Gamma}}\frac{\partial 
\Gamma_1}{\partial x_2}\frac{\partial \Gamma_2}{\partial x_1}+\\
&\quad +\frac{\partial\hat{f}_1}{\partial x_2}\Big|_{\tilde{\Gamma}}\frac{\partial\hat{f}_2}{\partial x_1}\Big|_{\tilde{\Gamma}}\frac{\partial 
\Gamma_2}{\partial x_2}\frac{\partial \Gamma_1}{\partial x_1}+\frac{\partial\hat{f}_1}{\partial x_2}\Big|_{\tilde{\Gamma}}\frac{\partial\hat{f}_2}{\partial 
x_2}\Big|_{\tilde{\Gamma}}\frac{\partial \Gamma_2}{\partial x_1}\frac{\partial \Gamma_2}{\partial x_2},
\end{align*}
therefore
\[
 \mathrm{d}(\hat{f}_1\circ\tilde{\Gamma})\wedge\mathrm{d}(\hat{f}_2\circ\tilde{\Gamma})=\Big(\frac{\partial\hat{f}_1}{\partial 
x_1}\frac{\partial\hat{f}_2}{\partial x_2}\Big|_{\tilde{\Gamma}}-\frac{\partial\hat{f}_1}{\partial x_2}\frac{\partial\hat{f}_2}{\partial 
x_1}\Big|_{\tilde{\Gamma}} \Big)\Big(\frac{\partial \Gamma_1}{\partial x_1}\frac{\partial \Gamma_1}{\partial x_2}-\frac{\partial 
\Gamma_1}{\partial x_2}\frac{\partial \Gamma_1}{\partial x_1} \Big)\mathrm{d}x_1\wedge\mathrm{d}x_2.
\]
be $\hat{f}_i\circ\tilde{\Gamma}$ constant in $t$ implies that if $\mathrm{d}\tilde{f}_1\wedge\mathrm{d}\tilde{f}_2\equiv 0$ then 
$\mathrm{d}(\hat{f}_1\circ\tilde{\Gamma})\wedge\mathrm{d}(\hat{f}_2\circ\tilde{\Gamma})\equiv0$, the former was restricted to $\{x_3=1\}$ and the later take values on the saturate of 
a small transverse section $\Sigma$ contained in $\{x_3=1\}$, as can bee see in \cite{Reis} (Proposition 1.) or Lemma \ref{lemA-dim3}, sat$\Sigma$ contains a neighborhood of 
the separatrices, which means that $\mathrm{d}\hat{f}_1\wedge\mathrm{d}\hat{f}_2\equiv 0$ and this is a contradiction.\\
With this in mind, by Proposition \ref{ribón}, we have that $\mathrm{Hol}(\F(\X),S,\Sigma)$ is periodic because it preserves $\{\tilde{f}_1,\tilde{f}_2\}$ and, its generated by one
germ of diffeomorphism. Therefore, the Theorem \ref{existence} implies that $\F(\X)$ has a holomorphic first integral.
\end{proof}
As for arbitrary dimension we have:
\begin{theorem}\label{theoB}
Let $\F(\X)$ be the germ of a holomorphic foliation with $\X\in\mathrm{Gen}\left(\mathfrak{X}(\C^n,0)\right)$ satisfying condition ($\star$), if $\F(\X)$ possesses a formal first
integral then it also possesses a holomorphic one.
\end{theorem}
\begin{proof}[Proof of Theorem \ref{theoB}]
The proof goes on as the previous one but now we use Theorem 5 in \cite{Reb-Reis} which needs the condition $(\star)$.
\end{proof}
\chapter{Vector fields and Darboux's Theorem}\label{Chap4}
\section{Preliminaries}
Be $\F$ a foliation by curves in $\C P(n)$ and $L$ a leaf of $\F$.
\begin{defi}
 We say that $L$ \emph{is algebraic} if the closure $\overline{L}$ of $L$ in $\C P(n)$, is an algebraic subset of dimension $1$, i.e., an algebraic curve. In this case, we also say 
that $\overline{L}$ is \emph{an algebraic solution of $\F$}.
\end{defi}
\begin{rem}
 Be $\F$ a foliation in $\C P(n)$, whose singularities are isolated. Then, a leaf $L$ of $\F$ is an algebraic solution, if and only if, $\overline{L}$ is obtained from $L$  
by the adjunction of the singularities of $\F$ to which $L$ is adherent. 
\end{rem}
\begin{theorem}[Darboux's Theorem]\label{darboux}
 Let $\F$ be a foliation in $\C P(2)$ who possesses infinitely many algebraic solutions. Then $\F$ admits a rational first integral.
\end{theorem}
\section{Vector fields with infinitely many invariant hypersurfaces}
\subsection{Homogeneous case}
\begin{defi}
  Let $\X\in\mathcal{X}(\C^3,0)$, we say that $\X$ is \emph{homogeneous of degree $\nu$} if $\X(x)=a_\nu(x)\frac{\partial}{\partial x_1}+b_\nu(x)\frac{\partial}{\partial 
x_2}+c_\nu(x)\frac{\partial}{\partial x_3}$ where $a_\nu,b_\nu$ and $c_\nu$ are homogeneous polynomials with same degree $\nu$ and without common factors.
\end{defi}
Note that if $\X$ is homogeneous of degree $\nu$ then $\X(\lambda x)=\lambda^{\nu}\X(x)$ for every $\lambda\in\C^*$, intuitively this means that along the line $\lambda x$ the vector 
field $\X$ points in the same direction allowing us to define a vector field $\tilde{\X}$ in 
the projective plane $\C P(2)$ as follows,\par
Remember that the usual differential structure of $\C P(2)$ is given by the atlas $\{(U_i,\varphi_i)\}_{i=1}^3$ where 
$U_i=\{[x_1;x_2;x_3]\in\C P(2)\,|\,x_i\neq 0\}$ and 
\begin{equation*}\begin{split}                 
  \varphi_1([x_1;x_2;x_3])=\Big(\frac{x_2}{x_1},\frac{x_3}{x_1}\Big)=(x,y),\\
  \varphi_2([x_1;x_2;x_3])=\Big(\frac{x_1}{x_2},\frac{x_3}{x_2}\Big)=(u,v),\\
  \varphi_3([x_1;x_2;x_3])=\Big(\frac{x_1}{x_3},\frac{x_2}{x_3}\Big)=(s,r).
\end{split}
\end{equation*}
Consider the projection
\begin{align*}
   \Pi:\C^3 \to\C P(2): (x_1,x_2,x_3)\to[(x_1;x_2;x_3)]=\{\lambda(x_1,x_2,x_3)\,|\,\lambda\in\C^*\}
\end{align*}
that in the first chart is written as $\Pi_1(x_1,x_2,x_3)=\varphi_1\circ\Pi(x_1,x_2,x_3)=(x,y)$. Putting all of this together, $\tilde{\X}$ in the first chart is: 
\begin{align*}
  \tilde{\X}_1(x,y)&=\Pi_1^*\X(x,y)\big|_{x_1=1}=\Big\{{\mathrm{d}\Pi_1}_{\Pi_1^{-1}(x,y)}\X\big(\Pi_1^{-1}(x,y)\big)\Big\}_{x_1=1},\\ 
   &=\begin{bmatrix}
		          -x&1&0\\
			  -y&0&1\\
		  \end{bmatrix}\X(1,x,y),
\end{align*}\vspace{-0.7cm}
\begin{align*}
	\tilde{\X}_1(x,y)=\big(b_\nu(1,x,y)-xa_\nu(1,x,y)\big)\frac{\partial}{\partial x}+\big(c_\nu(1,x,y)-ya_\nu(1,x,y)\big)\frac{\partial}{\partial y},		  
\end{align*}
in the same way for the other two charts. 
\begin{theorem}\label{homogeneous}
Let $\X$ be a germ of homogeneous vector field in $0\in\C^3$. Suppose that $\X$ leaves invariant infinitely many hypersurfaces passing through $0$ and in 
general position. Then, there exist a rational map $f:\C P(2)\to \C P(1)$ that is $\F(\X)$-\emph{invariant}  (i. e., $\X(f)\equiv 0$) this map is also call it a weak first integral 
of 
$\F(\X)$.
\end{theorem}
\begin{proof}
The idea of the proof is to send the vector field to the complex projective space and show that it defines there a foliation with infinitely many algebraic leaves, then we use 
Darboux's Theorem \ref{darboux} to obtain a first integral for this vector field which is a weak first integral for the original one.\par 
 Suppose that $S:=\{g=0\}$, for an irreducible $g\in \mathcal{M}_3$, is an $\X$-invariant hypersurface which is equivalent to said that $g$ \emph{divides} 
$\X(g)$, noted as $g\,\big|\,\X(g)$, to see this if $x_0\in S$ and $\phi(T)$ is the integral curve of the 
vector field $\X$ with $\phi(0)=x_0$ defined in a neighborhood of $0\in\C$ then, 
\[
\begin{cases}
 g(\phi)=0,\\
 \X(\phi(T))=\phi'(T),
\end{cases}
\hfill\text{together they imply that}\quad\X(g)(\phi)=0.
\]
Therefore, as $\X(g)(\cdot)$ is a holomorphic function which is null restricted to $S$, it can be written as
\begin{equation}\label{divides}
 \X(g)(\cdot)=g(\cdot)h(\cdot),
\end{equation}
where $h\in\mathcal{O}_3$.\\
Remember that if $\kappa$ is the order of $g$ then $g=g_\kappa+g_{\kappa+1}+\cdots$ where $g_m$ is a homogeneous polynomial of degree $m$, thus by the linearity of $\X$ as a 
derivation operator we have that 
\[
 \X(g)=\X(g_\kappa)+\X(g_{\kappa+1})+\cdots,
\]
is also a sum of homogeneous polynomials, $\X(g_\kappa)$ is homogeneous of order $\nu+\kappa-1$, $\X(g_{\kappa+1})$ is homogeneous of order $\nu+\kappa$, etc., being $\nu$ the 
order of $\X$ as before. Obviously $h$ in \eqref{divides} can also be written as a sum of homogeneous polynomials and the degree of the first not null of them (the order 
of $h$) necessarily is $\nu-1$ by \eqref{divides}. Using this, \eqref{divides} can be rewritten in the following way, 
\begin{align*}
 \X(g_\kappa)+\X(g_{\kappa+1})+\cdots&=(g_\kappa+g_{\kappa+1}+\cdots)(h_{\nu-1}+h_\nu+\cdots),\\
			      &=g_\kappa h_{\nu-1}+\dots,
\end{align*}
which implies, by comparing the degree of the terms in both sides, that
\[
 \X(g_\kappa)=g_\kappa h_{\nu-1},
\]
in other words $g_\kappa\,|\, \X(g_\kappa)$, thereby $S_\kappa:=\{g_\kappa=0\}$ is an $\X$-invariant algebraic hypersurface.\\
Next, as we mention previously the homogeneity of $\X$ can be used to define a vector field $\tilde{\X}$ in $\C P(2)$, the same can be done with $g_\kappa$ and define a function 
$\tilde{g}_\kappa$ in $\C P(2)$ as follows,
\begin{align*}
 \tilde{g}_\kappa(x,y)&=\Pi_1^*g_\kappa|_{x_1=1},\\
		      &=g_\kappa(\Pi_1^{-1}(x,y))|_{x_1=1},\\
		      &=g_\kappa(1,x,y)
\end{align*}
analogously in the other two charts. Let us see that $\tilde{g}_\kappa\,|\, \tilde{\X}(\tilde{g}_\kappa)$, first we use the equality 
$g_\kappa(x_1,x_2,x_3)=x_1^{\kappa}g_\kappa(1,x_2/x_1,x_3/x_1)=x_1^{\kappa}g_\kappa(1,x,y)$ to calculate $\triangledown g_\kappa(1,x,y)$ in terms of $x_1,x_2$ and $x_3$, as below,
\begin{align*}
 \frac{\partial g_\kappa}{\partial x_1}&=\kappa x_1^{\kappa-1}g_\kappa+x_1^\kappa\Big(\frac{\partial g_\kappa}{\partial x}\frac{d x}{d x_1}+\frac{\partial g_\kappa}{\partial 
y}\frac{d y}{d x_1}\Big),\\
				       &=\kappa x_1^{\kappa-1}g_\kappa+x_1^{\kappa-1}\Big(-x\frac{\partial g_\kappa}{\partial x}-y\frac{\partial g_\kappa}{\partial 
y}\Big),\\
 \frac{\partial g_\kappa}{\partial x_2}&=x_1^\kappa\Big(\frac{\partial g_\kappa}{\partial x}\frac{d x}{d x_2}+\frac{\partial g_\kappa}{\partial 
y}\frac{d y}{d x_2}\Big),\\
				       &=x_1^{\kappa-1}\frac{\partial g_\kappa}{\partial x},\\
 \frac{\partial g_\kappa}{\partial x_3}&=x_1^\kappa\Big(\frac{\partial g_\kappa}{\partial x}\frac{d x}{d x_3}+\frac{\partial g_\kappa}{\partial 
y}\frac{d y}{d x_3}\Big),\\
				       &=x_1^{\kappa-1}\frac{\partial g_\kappa}{\partial y},
\end{align*}
if we set $x_1=1$ they become,
\[
 \frac{\partial g_\kappa}{\partial x_1}=\kappa g_\kappa+\Big(-x\frac{\partial g_\kappa}{\partial x}-y\frac{\partial g_\kappa}{\partial 
y}\Big),\quad
 \frac{\partial g_\kappa}{\partial x_2}=\frac{\partial g_\kappa}{\partial x},\quad
 \frac{\partial g_\kappa}{\partial x_3}=\frac{\partial g_\kappa}{\partial y},
\]
second, keep in mind that $\X(g_\kappa)=a_\nu\frac{\partial g_\kappa}{\partial x_1}+b_\nu\frac{\partial g_\kappa}{\partial 
x_2}+c_\nu\frac{\partial g_\kappa}{\partial x_3}=g_\kappa h_\nu$ in particular for $x_1=1$ now, $\tilde{g}_\kappa\,|\, \tilde{\X}(\tilde{g}_\kappa)$ is consequence of the previous 
considerations, 
\begin{align*}
 \tilde{\X}(\tilde{g}_\kappa)&=\begin{bmatrix}
				 -x&1&0\\
				 -y&0&1\\
			      \end{bmatrix}\X(1,x,y)\cdot\triangledown g_\kappa(1,x,y),\\
			      &=\big(-xa_\nu+b_\nu\big)\frac{\partial g_\kappa}{\partial x}+\big(-ya_\nu+c_\nu\big)\frac{\partial g_\kappa}{\partial y},\\
			      &=a_\nu\Big(-x\frac{\partial g_\kappa}{\partial x}-y\frac{\partial g_\kappa}{\partial y}\Big)+b_\nu\frac{\partial g_\kappa}{\partial 
x}+c_\nu\frac{\partial g_\kappa}{\partial y},\\
			      &=-\kappa a_\nu g_\kappa+\Big(a_\nu\frac{\partial g_\kappa}{\partial x_1}+b_\nu\frac{\partial g_\kappa}{\partial 
x_2}+c_\nu\frac{\partial g_\kappa}{\partial x_3}\Big),\\
			      &=-\kappa a_\nu g_\kappa+g_\kappa h_\nu,\\  
\tilde{\X}(\tilde{g}_\kappa)&=g_\kappa\big(-\kappa a_\nu+h_\nu\big),
\end{align*}
where all the functions are evaluated in $(1,x,y)$.\par 
Thus, $\{\tilde{g}_\kappa=0\}$ is an algebraic curve $\tilde{\X}$-invariant. The same argument is valid with any of the infinitely many $\X$-invariant hypersurfaces and the fact 
that there are infinitely many in general position implies that there exist infinitely many algebraic curves $\tilde{\X}$-invariant, then by 
Darboux's Theorem, $\tilde{\X}$ posseses a rational first integral $f:\C P(2)\to \C P(1)$. Only remains to see that $f$ is $\F(\X)$-invariant, this is equivalent to verify that 
$\X(f)\equiv 0$, which is the next an final step in the proof.\par
We can think $f$ as a function in $\C^3$ constant along the directions $f(\lambda x)=f(x)$ in other words, homogeneous of order $0$. So, as we did before with $g_\kappa$, $f$ can be 
written as $f(x_1,x_2,x_3)=f(1,x_2/x_1,x_3/x_1)=f(1,x,y)$ and by derivation,     
\[
 \frac{\partial f}{\partial x_1}=-\frac{x}{x_1}\frac{\partial f}{\partial x}-\frac{y}{x_1}\frac{\partial f}{\partial 
y},\quad
 \frac{\partial f}{\partial x_2}=\frac{1}{x_1}\frac{\partial f}{\partial x},\quad
 \frac{\partial f}{\partial x_3}=\frac{1}{x_1}\frac{\partial f}{\partial y},
\]
using that $\tilde{\X}_1(f)=(-xa_\nu+b_\nu)\frac{\partial f}{\partial x}+(-ya_\nu+c_\nu)\frac{\partial f}{\partial y}\equiv 0$ where all the functions are evaluated in $(1,x,y)$, we 
can calculate
\begin{align*}
    \X(f)&=a_\nu(x_1,x_2,x_3)\frac{\partial f}{\partial x_1}+b_\nu(x_1,x_2,x_3)\frac{\partial f}{\partial x_2}+c_\nu(x_1,x_2,x_3)\frac{\partial f}{\partial x_3},\\
	 &=x_1^{\nu}\Big(a_\nu(1,x,y)\frac{\partial f}{\partial x_1}+b_\nu(1,x,y)\frac{\partial f}{\partial x_2}+c_\nu(1,x,y)\frac{\partial f}{\partial x_3}\Big), \\
	 &=x_1^{\nu-1}\Big(a_\nu(1,x,y)\big( -x\frac{\partial f}{\partial x}-y\frac{\partial f}{\partial y}\big)+b_\nu(1,x,y)\frac{\partial f}{\partial x}\\
	 &\hspace{5,5cm}+c_\nu(1,x,y)\frac{\partial f}{\partial y}\Big),\\
	 &=x_1^{\nu-1}\Big((-xa_\nu+b_\nu)\frac{\partial f}{\partial x}+(-ya_\nu+c_\nu)\frac{\partial f}{\partial y}\Big)\equiv0,\\
    \X(f)&\equiv0.\qedhere
\end{align*}
\end{proof}\noindent
In order to conclude the homogeneous case is important to note that the previous method does not produce two weak first integrals transversally independent, because both 
of them are first integrals of $\tilde{\X}$ then in $\C^3$ they have the same level sets.  
\subsection{Generalities on blow-ups.}
Suppose that $\X(x)=a(x_1,x_2,x_3)\frac{\partial}{\partial x_1}+b(x_1,x_2,x_3)\frac{\partial}{\partial x_2}+c(x_1,x_2,x_3)\frac{\partial}{\partial x_3}$, where
$a,b,c\in\mathcal{O}_3$ are given by $a(x)=\sum_{|I|\geq p_1} a_Ix^I,\ 
b(x)=\sum_{|J|\geq p_2} b_Jx^J$ and $c(x)=\sum_{|K|\geq p_3} c_Kx^K$ . If $\varphi_1$ is the first chart of the blow-up, we note 
$E\circ\varphi^{-1}_1(z_1,z_2,z_3)=(z_1,z_1z_2,z_1z_3)$ simply by $E_1(z)$,  $a(E_1(z))$ by $a(z)$ and in the same way $b(z),\ c(z)$. Observe that in this chart the 
divisor, $D:=E^{-1}(0)=\C P(2)$, is given by $\{z_1=0\}$.\par
Using this notation we calculate $\tilde{\X}(z)=(\mathrm{d}E_1^{-1})_{E_1(z)}\X\big(E_1(z)\big)$,
\[ \mathrm{d}E_1=\begin{bmatrix}
              1&0&0\\
              z_2&z_1&0\\
              z_3&0&z_1\\
             \end{bmatrix},\quad 
   \mathrm{d}E_1^{-1}=\frac{1}{z_1^2}\begin{bmatrix}
              z_1^2&0&0\\
              -z_1z_2&z_1&0\\
              -z_1z_3&0&z_1\\
             \end{bmatrix},           
            \]
thus, 
\[\tilde{\X}(z)=\frac{1}{z_1^2}\begin{bmatrix}
              z_1^2&0&0\\
              -z_1z_2&z_1&0\\
              -z_1z_3&0&z_1\\
             \end{bmatrix}\cdot\begin{bmatrix} a(z)\\ b(z)\\ c(z)\end{bmatrix}, \]
\begin{align*}
\tilde{\X}(z)&=a(z)\frac{\partial}{\partial z_1}+\frac{1}{z_1}(-z_2a(z)+b(z))\frac{\partial}{\partial z_2}
+\frac{1}{z_1}(-z_3a(z)+c(z))\frac{\partial}{\partial z_3}\\           
&=\big(z_1^{\nu}j^{\nu}a(1,z_2,z_3)+z_1^{\nu+1}(\dots)\big)\frac{\partial}{\partial z_1}+\\
&\quad\ \big(-z_2z_1^{\nu-1}j^{\nu}a(1,z_2,z_3)+z_1^{\nu-1}j^{\nu}b(1,z_2,z_3)+z_1^{\nu}(\dots)\big)\frac{\partial}{\partial z_2}+\\
&\quad\ \big(-z_3z_1^{\nu-1}j^{\nu}a(1,z_2,z_3)+z_1^{\nu-1}j^{\nu}c(1,z_2,z_3)+z_1^{\nu}(\dots)\big)\frac{\partial}{\partial z_3},\\          
\tilde{\X}(z)&=z_1^{\nu}j^{\nu}a(1,z_2,z_3)\frac{\partial}{\partial z_1}+\\
&\quad\ z_1^{\nu-1}\big(-z_2j^{\nu}a(1,z_2,z_3)+j^{\nu}b(1,z_2,z_3)\big)\frac{\partial}{\partial z_2}+\\
&\quad\ z_1^{\nu-1}\big(-z_3z_1^{\nu-1}j^{\nu}a(1,z_2,z_3)+j^{\nu}c(1,z_2,z_3)\big)\frac{\partial}{\partial z_3}+z_1^{\nu}(\dots),     
\end{align*}
where $j^{\nu}(\cdot)$ means the $\nu$-jet and $\nu=\mathrm{min}\{p_1,p_2,p_3\}$ then, supposing that $x_2j^\nu a\neq x_1j^\nu b$ or $x_3j^\nu a\neq x_1j^\nu c$ (i.e., $0$ is a not 
dicritic singularity \cite{VFields-Cano}) in that case we can define in the first chart of $D$
\[\tilde{X}_D(z_2,z_3):=\big((z_1^{\nu-1})^{-1}\tilde{\X}(z)\big)_{z_1=0}\] and, we
have that
\begin{equation}\label{firstchart}
 \begin{split}
  \tilde{X}_D(z_2,z_3)=\big(-z_2j^{\nu}a(1,z_2,z_3)+j^{\nu}b(1,z_2,z_3)\big)\frac{\partial}{\partial z_2}+\\
   \big(-z_3j^{\nu}a(1,z_2,z_3)+j^{\nu}c(1,z_2,z_3)\big)\frac{\partial}{\partial z_3},  
 \end{split}
\end{equation}
to write $\tilde{X}_D$ in the others chart, that we will note $\tilde{X}_D(s,t)$ and $\tilde{X}_D(u,v)$, for simplicity, remember that\newpage
\begin{wrapfigure}{l}{0.5\textwidth}
  \vspace{-20pt}
  \begin{center}
    \resizebox{5cm}{!}{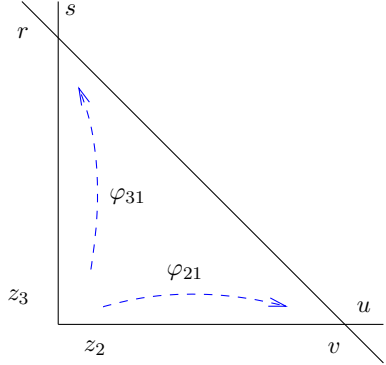}
    \caption{Change of Charts}
    \label{explo7}
  \end{center}
  \vspace{-90pt}
\end{wrapfigure}
where, 
\begin{align*}
\varphi_{21}(z_2,z_3)&=(u,v)\\
 u&=1/z_2\\
 v&=z_3/z_2,
\end{align*}and
\begin{align*}
 \varphi_{31}(z_2,z_3)&=(r,s)\\
 r&=z_2/z_3\\
 s&=1/z_3,\\
\end{align*}\vspace{0.52cm}\\
hence,
\begin{align*}
\tilde{X}_D(u,v)&=u^{\nu-1}\mathrm{d}\varphi_{21}\tilde{X}_D(\varphi_{21}^{-1}(u,v))\quad\text{and}\\ 
\tilde{X}_D(r,s)&=s^{\nu-1}\mathrm{d}\varphi_{31}\tilde{X}_D(\varphi_{31}^{-1}(r,s)), 
\end{align*}
using that 
\[ \mathrm{d}\varphi_{21}=\begin{bmatrix}
                        -u^2 & 0\\
                         -uv & u\\                         
                        \end{bmatrix}\ \text{ and }\  
 \mathrm{d}\varphi_{31}=\begin{bmatrix}
                          s & -rs\\
                          0 & -s^2\\
                        \end{bmatrix},              
\]
we have,
\begin{align*} 
\tilde{X}_D(u,v)&=u^\nu\Big(-uj^{\nu}b\Big(1,\frac{1}{u},\frac{v}{u}\Big)+j^{\nu}a\Big(1,\frac{1}{u},\frac{v}{u}\Big)\Big)\frac{\partial}{\partial u}+\\
&\quad\ u^\nu\Big(-vj^{\nu}b\Big(1,\frac{1}{u},\frac{v}{u}\Big)+j^{\nu}c\Big(1,\frac{1}{u},\frac{v}{u}\Big)\Big)\frac{\partial}{\partial v},\\
\end{align*}
and
\begin{align*} 
\tilde{X}_D(r,s)&=s^\nu\Big(-rj^{\nu}c\Big(1,\frac{r}{s},\frac{1}{s}\Big)+j^{\nu}b\Big(1,\frac{r}{s},\frac{1}{s}\Big)\Big)\frac{\partial}{\partial r}+\\
&\quad\ s^\nu\Big(-sj^{\nu}c\Big(1,\frac{r}{s},\frac{1}{s}\Big)+j^{\nu}a\Big(1,\frac{r}{s},\frac{1}{s}\Big)\Big)\frac{\partial}{\partial s},\\
\tilde{X}_D(r,s)&=\Big(-rj^{\nu}c(s,r,1)+j^{\nu}b(s,r,1)\Big)\frac{\partial}{\partial r}+\\
&\quad\ \Big(-sj^{\nu}c(s,r,1)+j^{\nu}a(s,r,1)\Big)\frac{\partial}{\partial s},
\end{align*}
Observe that $\tilde{X}_D(z_2,z_3)$ is a polynomial vector field of degree $\leq\nu+1$ leaving $D$ invariant.
\subsection{General case}
\begin{lem}
If a vector field $\X\in\mathcal{X}(\C^3,0)$ leaves invariant a hypersurface passing through $0$, then its first jet $\X_\nu$  leaves invariant an algebraic hypersurface passing 
through $0$.  
\end{lem}
\begin{proof}
 The argument is similar to the one in the first part of the demonstration of Theorem \ref{homogeneous}. Let $S=\{g=0\}$, for $g\in\mathcal{M}_3$ irreducible, be a $\X$-invariant 
hypersurface, then it exist $h\in\mathcal{M}_3$ such that $\X(g)=gh$. The three of them, $\X,g$ and $h$ can be written as a sum of homogeneous terms,
\begin{align*}
 \X&=\X_\nu+\X_{\nu+1}+\cdots,\\
 g&=g_\kappa+g_{\kappa+1}+\cdots,\\
 h&=h_{\nu-1}+h_\nu+\cdots,
\end{align*}
 the equality $\X(g)=gh$ implies that the order of $h$ is $\nu-1$, and by comparing both sides of
\[
  \X_\nu(g_\kappa+g_{\kappa+1}+\cdots)+\X_{\nu+1}(g_\kappa+\cdots)+\cdots=(g_\kappa+\cdots)(h_{\nu-1}+\cdots),
\]
we get that $\X_\nu(g_\kappa)=g_\kappa h_{\nu-1}$.
\end{proof}
In what follows, we note by $\tilde{\X}$ the push-back of the vector field $\X$ by the blow-up $E:\tilde{\C^3}\to\C^3$ at the origin and $\tilde{X}_D$ its restriction to the divisor 
and we have,
\begin{pro}
  Let $\F(\X)$ be the germ of a holomorphic foliation with $\X\in\mathfrak{X}(\C^3,0)$ having a isolated not dicritic singularity at $0$. If there exist infinitely many 
$\X$-invariant analytic 
hypersurfaces passing through $0$ and in general position then $\tilde{X}_D$ possesses a rational first integral.
 \end{pro}
\begin{proof}
  The previous lemma, together with Theorem \ref{homogeneous} implies that $\X_\nu$ possesses a weak first integral, now remembering the previous section, $\tilde{\X}$ in the first 
chart of the blow-up is given by $\tilde{\X}(z)=(\mathrm{d}E_1^{-1})_{E_1(z)}\X(E_1(z))$ where 
  $E_1(z_1,z_2,z_3)=(z_1,z_1z_2,z_1z_3)$, $E_1^{-1}(x_1,x_2,x_3)=(x_1,x_2/x_1,x_3/x_1)$, $z_1=x_1,\ z_2=x_2/x_1,\ z_3=x_3/x_1$ and
\[
 (\mathrm{d}E_1^{-1})_x=\begin{bmatrix}
			   1&0&0\\
			  -x_2/x_1^2&1/x_1&0\\
			  -x_3/x_1^2&0&1/x_1\\
			\end{bmatrix},\quad 
(\mathrm{d}E_1^{-1})_{E_1(z)}=\frac{1}{z_1}\begin{bmatrix}
			   z_1&0&0\\
			  -z_2&1&0\\
			  -z_3&0&1\\
			\end{bmatrix}, 
\]
if $\X(x)=a(x)\frac{\partial}{\partial x_1}+b(x)\frac{\partial}{\partial x_2}+c(x)\frac{\partial}{\partial x_3}$ then 
\begin{align*}
 \tilde{\X}(z)&=a(z)\frac{\partial}{\partial z_1}+\frac{1}{z_1}\big(-z_2a(z)+b(z)\big)\frac{\partial}{\partial z_2}+\frac{1}{z_1}\big(-z_3a(z)+c(z)\big)\frac{\partial}{\partial 
z_3},\\
	      &=z_1^{\nu}a_\nu(z)\frac{\partial}{\partial z_1}+z_1^{\nu-1}\big(-z_2a_\nu(z)+b_\nu(z)\big)\frac{\partial}{\partial 
z_2}+\\
	      &\hspace{2,25cm}+z_1^{\nu-1}\big(-z_3a_\nu(z)+c_\nu(z)\big)\frac{\partial}{\partial z_3}+z_1^\nu(\dots)
\end{align*}
and $\tilde{\X}_D(z)=\big[(z_1^{\nu-1})^{-1}\tilde{\X}(z)\big]_{z_1=0}$ thus,
\[
 \tilde{\X}_D(z)=\big(-z_2a_\nu(z)+b_\nu(z)\big)\frac{\partial}{\partial z_2}+\big(-z_3a_\nu(z)+c_\nu(z)\big)\frac{\partial}{\partial z_3},
\]
Now, as we mention before, there exist $f:\C P(2)\to\C P(1)$ such that $\X_\nu(f)\equiv 0$, i.e., $a_\nu(x)\frac{\partial f}{\partial x_1}+b_\nu(x)\frac{\partial f}{\partial 
x_2}+c_\nu(x)\frac{\partial f}{\partial x_3}\equiv0$, and we proceed as in the end of the proof of Theorem \ref{homogeneous},
\begin{align*}
 \tilde{\X}_D(f)&=\big(-z_2a_\nu(z)+b_\nu(z)\big)\frac{\partial f}{\partial z_2}+\big(-z_3a_\nu(z)+c_\nu(z)\big)\frac{\partial f}{\partial z_3},\\
                &=-z_2a_\nu(z)\frac{\partial f}{\partial z_2}-z_3a_\nu(z)\frac{\partial f}{\partial z_3}+b_\nu(z)\frac{\partial f}{\partial z_2}+c_\nu(z)\frac{\partial f}{\partial 
z_3},\\
		&=x_1\Big(a_\nu(z)\frac{\partial f}{\partial x_1}+b_\nu(z)\frac{\partial f}{\partial x_2}+c_\nu(z)\frac{\partial f}{\partial x_3}\Big),\\
\tilde{\X}_D(f)&\equiv0.
\end{align*}
In the part above we use the following notation $a_\nu(x)=a_\nu(x_1,x_2,x_3)=z_1^{\nu}a_\nu(1,z_2,z_3)=z_1^{\nu}a_\nu(z)$, and that $f(x_1,x_2,x_3)=f(1,z_2,z_3)$ which implies by 
derivation,
\[
 \frac{\partial f}{\partial x_1}=-\frac{z_2}{x_1}\frac{\partial f}{\partial z_2}-\frac{z_3}{x_1}\frac{\partial f}{\partial 
z_3},\quad
 \frac{\partial f}{\partial x_2}=\frac{1}{x_1}\frac{\partial f}{\partial z_2},\quad
 \frac{\partial f}{\partial x_3}=\frac{1}{x_1}\frac{\partial f}{\partial z_3}.
\]
\end{proof}
\begin{rem}[about condition ($\star$) ]
 The condition ($\star$) was defined in Definition \ref{star} and is it possible to choose a vector $v$ such that Re$\Big(\frac{\lambda_i}{v}\Big)$ has different sign for the 
eigenvalue $\lambda_i$ that can be separated. 
 \begin{figure}[htbp]
 \begin{center}
    \resizebox{5cm}{!}{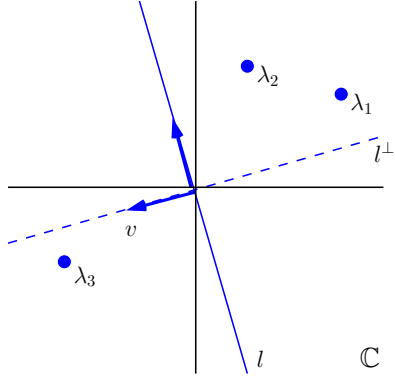}
     \caption{Condition $(\star)$ with $l$ the line separating $\lambda_3$.}
      \label{Vchap3}
\end{center}
\end{figure}
\end{rem}
\begin{theorem}
 Let $\F(\X)$ be the germ of a holomorphic foliation with $\X\in\mathrm{Gen}\left(\mathfrak{X}(\C^3,0)\right)$ and satisfying condition ($\star$). Then $\F(\X)$ has a holomorphic 
first integral if, and only if, the leaves of $\F(\X)$ are closed off the singularity and there exist infinitely many $\X$-invariant analytic hypersurfaces passing through $0$ and in 
general position.
\end{theorem}
 \begin{proof}
We are considering $\X\in\mathrm{Gen}(\mathfrak{X}(\C^n , 0))$, and by definition (see \eqref{gen-vf}), after a change of coordinates it can be written in the form
\begin{equation}\label{theform}
  \X(x) = \lambda_1x_1 \big(1 + a_1(x)\big)\frac{\partial}{\partial x_1}+\lambda_2x_2 \big(1 + a_2(x)\big)\frac{\partial}{\partial x_2}+\lambda_3x_3 \big(1 + 
  a_3(x)\big)\frac{\partial}{\partial x_3}
\end{equation}
this vector field is in the conditions of Theorem \ref{existence} just remaining to prove that the holonomy respect to the distinguished axis of $\X$ (noted $S_\X$ as before)  
is periodic, remember that $S_\X$ is the invariant manifold associated to the eigenvalue that can be separated, in this case assume that is $\lambda_3$. We can calculate 
$\mathrm{Hol}(\F(\X), S_{\X}, \Sigma)$ 
\begin{figure}[htbp]
 \begin{center}
    \resizebox{10cm}{!}{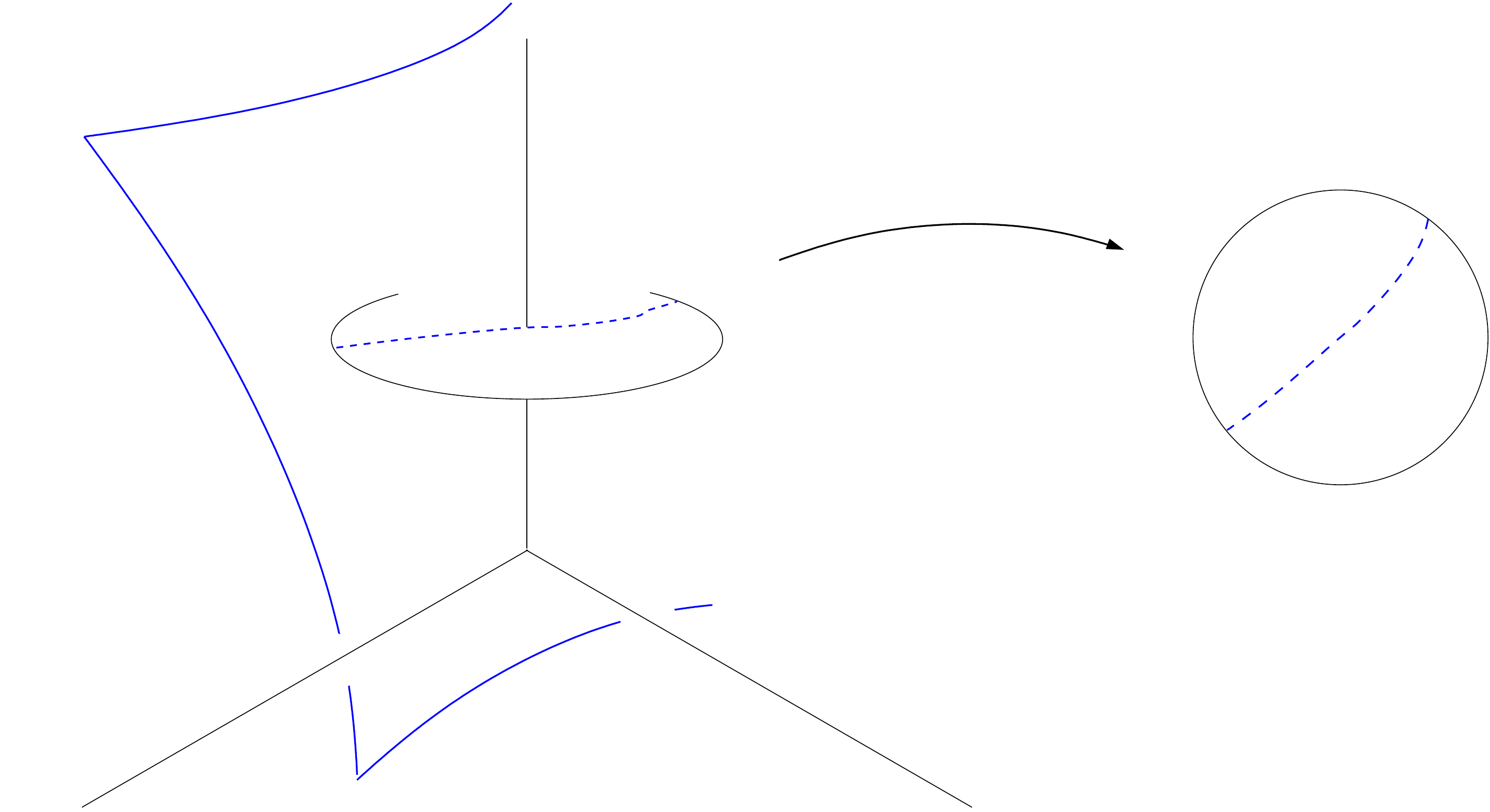}
     \caption{Holonomy of $S_{\X}$ }
      \label{Vchap1}
\end{center}
\end{figure}
taking a small transversal section $\Sigma$ to $S_\X$ in some point $z_0$ close to the origin and diffeomorphic to a ball in $\C^2$.\\
Observe first that if $z_0$ is close enough to the origin the saturate of $\Sigma$ together with the hyperplane $\{x_3=0\}$ contains a neighborhood of the origin (see 
Proposition 1 \cite{Reis}), this means that every $\X$-invariant hypersurfaces distinct to $\{x_3=0\}$ necessarily cuts $\Sigma$ because as it contains $0$ then it cuts the 
saturate of $\Sigma$ and by its $\X$-invariance it contains also the leaves coming through $\Sigma$. Furthermore, we can guarantee that infinitely many not only cut $\Sigma$ 
but contain the $x_3$ axis, in  order to see this take $S=\{g=0\}$ a $\X$-invariant hypersurface given by the zero set of $g(x)=\Sigma_{|I|\geq\nu}b_{I}x^I$ then 
$\X(g)(x)=g(x)h(x)$, where $h(x)=\Sigma_{I}c_{I}x^I$, using \eqref{theform} to write this equation in therms of the series, we have 
\begin{align*}
 \Sigma_{|I|\geq\nu}&\big[\lambda_1i\big(1+a_1(x)\big)+\lambda_2j\big(1+a_2(x)\big)+\lambda_3k\big(1+a_3(x)\big)\big]b_Ix^I=\\
 &\Big(\Sigma_{|I|\geq\nu}b_{I}x^I\Big)\Big(\Sigma_{I}c_{I}x^I\Big),
\end{align*}
making $x_2=x_3=0$ we get
\begin{align*}
 \Sigma_{i\geq i_0}\lambda_1i\big(1+a_1(x_1,0,0)\big)b_{i,0,0}x_1^{i}=\Big(\Sigma_{i\geq i_0}b_{i,0,0}x_1^{i}\Big)\Big(\Sigma_{i}c_{i,0,0}x_1^i\Big),
\end{align*}
in a similar way for $x_1=x_2=0$ and $x_1=x_3=0$, and by comparing the first terms in both sides,
\begin{align*}
 \lambda_1i_0b_{i_0,0,0}&=b_{i_0,0,0}c_{0},\\
 \lambda_2j_0b_{0,j_0,0}&=b_{0,j_0,0}c_{0},\\
 \lambda_3k_0b_{0,0,k_0}&=b_{0,0,k_0}c_{0}.
\end{align*}
Remember that our intention is to show that $g(0,0,x_3)=\Sigma_{k\geq k_0}b_{0,0,k}x_3^{k}\equiv 0$ (because this implies that the $x_3$ axis belongs to $S$) for this is enough to 
show that $b_{0,0,k_0}$=0 because in theory it is the first not null term. If $c_0=0$ then 
$b_{0,0,k_0}=0$ given that $b_0=b_{0,0,0}=g(0)=0$, by hypothesis $0\in S$, then $k_0>0$ and we are done. If $c_0\neq 0$, suppose first that the three $b_{i_0,0,0}$, $b_{0,j_0,0}$ and 
$b_{0,0,k_0}$ are not zero then,
\[
 \lambda_1i_0=\lambda_2j_0=\lambda_3k_0,
\]
dividing by the vector $v$ as in fig. \ref{Vchap3} 
and comparing the real parts we have 
\[
 \mathrm{Re}\Big(\frac{\lambda_1}{v}\Big)i_0=\mathrm{Re}\Big(\frac{\lambda_2}{v}\Big)j_0=\mathrm{Re}\Big(\frac{\lambda_3}{v}\Big)k_0\quad (\rightarrow\leftarrow),
\]
this is a contradiction because $v$ can be chosen so that $\mathrm{Re}\big(\frac{\lambda_3}{v}\big)>0$ and the other two are negative. Hence, at least one of $b_{i_0,0,0}$, 
$b_{0,j_0,0}$ and $b_{0,0,k_0}$ has to be zero, the same analysis shows that $b_{i_0,0,0}\cdot b_{0,0,k_0}\neq 0$ or $b_{0,j_0,0}\cdot b_{0,0,k_0}\neq 0$ can not happen. Thus, any 
hypersurface not containing one of the axis $x_1$ or $x_2$ necessarily contains the axis $x_3$, this shows that infinitely many $\X$-invariant hypersurfaces  
cut $\Sigma$ forming $G$-invariants analytic curves (calling $G$ the holonomy map) as in fig. \ref{Vchap1}
in a such way that if we think in $\Sigma$ as a ball in $\C^2$, 
each one of those $G$-invariants curves contains $0$. \medskip\\    
Therefore $G$ generates a finite group according to Theorem \ref{BroN}, and this implies the existence of a holomorphic first integral for $\F(\tilde{\X})$ in 
some neighborhood of $0$.
\end{proof}
\chapter{Complete stability theorem for foliations with singularities}\label{Chap5} 
In this chapter is used the following result about closed leaves of holomorphic foliations,
\begin{theorem}[\cite{ghys1}]\label{teo_ghys}
 Let $\F$ be a holomorphic foliation (possibly singular) of codimension $1$ in a compact and connected complex manifold. Then $\F$ has a finite number of closed leaves unless 
it possesses a meromorphic first integral, in which case all the leaves are closed.     
\end{theorem}
to obtain a stability theorem (Theorem \ref{Brun&Ghys}) for a special kind of codimension 1 foliations with singularities in a compact, connected and complex analytic two dimensional 
variety. We will state the result of this chapter in Section \ref{mainresult} after some definitions.    
 \begin{defi}[\cite{Barth,gunning2}]
  A \emph{divisor} $\mathcal{D}$ on a compact complex manifold $M$, is a formal sum $\mathfrak{d}_p=\sum_jk_j\mathbf{V}_j$ where $k_j\in\mathbb{Z}$ and $\{\mathbf{V}_j\}_j$ 
is a 
locally finite sequence of irreducible hypersurfaces on $M$, where locally finite means that every point has a neighborhood which meets only finitely many $\mathbf{V}_j$'s.
 \end{defi}
\section{Holonomy and virtual holonomy groups}
 Let now $\F$ be a holomorphic foliation with (isolated)  singularities on a complex surface $M$ (we have in mind here, the result of a reduction of singularities process). Denote 
by $\mathrm{Sing}(\F)$ the singular set of $\F$. Given a leaf $L_0$ of $\F$ we choose any base point $p\in L_0\subset M \setminus \mathrm{Sing}(\F)$ and a transverse disc 
$\Sigma_p\subset M$ to $\F$ centered at $p$. The holonomy group of the leaf $L_0$ with respect to the disc $\Sigma_p$ and to the base point $p$ is image of the representation 
$\mathrm{Hol}\colon \pi_1(L_0,p) \to \mathrm{Diff}(\Sigma_p,p)$ obtained by lifting closed paths in $L_0$ with base point $p$, to paths in the leaves of $\F$,  starting at points 
$z\in \Sigma_p$, by means of a transverse fibration to $\F$ containing the disc $\Sigma_p$ (\cite{livCamNet}). Given a point $z \in \Sigma_p$ we denote the leaf through $z$ by $L_z$.
Given a closed path $\gamma \in \pi_1(L_0,p)$ we denote by $\tilde \gamma_z$ its lift to the leaf $L_z$ and starting (the lifted path) at the point $z$. Then the image of the 
corresponding holonomy map is $h_{[\gamma]}(z)=\tilde \gamma_z(1)$, i.e., the final point of the lifted path $\tilde \gamma_z$.
This defines a diffeomorphism germ map  $h_{[\gamma]} \colon (\Sigma_p, p) \to (\Sigma_p,p)$ and also a group homomorphism $\mathrm{Hol} \colon \pi_1(L_0,p) \to 
\mathrm{Diff}(\Sigma_p,p)$. The image $\mathrm{Hol}(\F,L_0,\Sigma_p,p)\subset \mathrm{Diff}(\Sigma_p,p)$ of such homomorphism is called the {\it holonomy group} of the leaf $L_0$ 
with 
respect to
$\Sigma_p$ and $p$. By considering any parametrization $z\colon (\Sigma_p,p) \to (\mathbb D,0)$ we may identify (in a non-canonical way) the holonomy group with a subgroup of 
$\mathrm{Diff}(\mathbb C,0)$. It is clear from the construction  that the maps in the holonomy group preserve the leaves of the foliation. Nevertheless, this property can be shared 
by 
a
larger group that may therefore contain more information about the foliation in a neighborhood of the leaf.
The {\it virtual holonomy group} of the leaf with respect to the transverse section $\Sigma_p$ and base point $p$ is defined as (\cite{cam-neto-sad}, \cite{camacho-scarduaasterisque})
\[
\mathrm{Hol}^{\mathrm{virt}} (\F, \Sigma_p,p)=\{ f \in\mathrm{Diff}
(\Sigma_p, p) \big| L_z =  L_{f(z)},
\forall z\in (\Sigma_p, p) \}
\]
The virtual holonomy group contains the holonomy group and consists of the map germs that preserve the leaves of the foliation. Fix now a germ of holomorphic foliation with a 
singularity at the origin $0\in \mathbb C^2$, with a  representative $\F(U)$ as above. Let $\Gamma$ be a separatrix of $\F$. By Newton-Puiseaux parametrization theorem, the topology 
of $\Gamma$  is the one of a disc. Further, $\Gamma\setminus \{0\}$  is biholomorphic to a punctured disc $\mathbb D^*= \mathbb D\setminus \{0\}$. In particular,  we may choose a 
loop $\gamma\in \Gamma\setminus\{0\}$ generating the (local) fundamental group $\pi_{1}(\Gamma\setminus\{0\})$. The corresponding holonomy map $h_{\gamma}$ is defined in terms of a 
germ of complex diffeomorphism at the origin of a local disc $\Sigma$ transverse to $\mathcal{F}$ and centered at a non-singular point $q\in \Gamma\setminus\{0\}$. This map is 
well-defined up to conjugacy by germs of holomorphic diffeomorphisms, and is generically referred to as \textit{local holonomy} of the separatrix $\Gamma$.

\section{Main result}\label{mainresult}
\begin{theorem}\label{Brun&Ghys}
 Be $\F$ a holomorphic foliation of codimension 1 on a compact, connected and complex analytic two dimensional variety $M$. If the following conditions are satisfied,
\begin{itemize}
 \item The virtual holonomy is finite,
 \item There exist a $\F$-invariant divisor $\mathcal{D}$ of $M$ containing the separatrices of a singularity $p$ of $\F$,
 \item Other singularities (if any) in a separatriz $L_i$ of $p$ are isolated, dicritical and $L_i$ meets a dicritical component of its 
resolution.
\end{itemize}
Then $\F$ has a meromorphic first integral.
\end{theorem}
\begin{proof}
  Suppose that  $\{L_i\}_i^r$ are the separatrices at $p$, we know that $r<\infty$ because $\{L_i\}_i^r\subset\mathcal{D}$, we also know that the number of singularities 
$p_j\in\{\overline{L}_i\}_i^r$ is finite because infinitely many would belong to one $L_i$ and they would accumulate, in contradiction with the hypothesis. Note $\tilde{M}$ the 
manifold obtained from $M$ after a resolution of the dicritical singularities $\{p_j\}_j$, and note $\tilde{\F}$ the associated foliation. Remember that, $\tilde{\F}$ coincides with 
$\F$ in $\tilde{M}\setminus D$, where $D$ is the union of the exceptional divisors $D_j$ one for each dicritic singularity, each $D_j$ is a finite union of projective 
lines, the singularities of $\tilde{\F}$ in $D$ are simples, and dicritic divisors in $D_j$ does not have singularities nor tangency points. Fixing a $i$, 
if $p_j$ is a dicritic singularity in $L_i$, by hypothesis $L_i$ meets a dicritical component of its resolution, locally in one chart $U_j$, of that component, the 
restriction of $L_i$ is one of the coordinates axis, the leaves are transversal to the other one and the induced foliation $\tilde{\F}_{U_j}$ is not singular.\\
Now, take a small neighborhood $U_0$ of $p$ and consider the induced foliation $\tilde{\F}_{U_0}$ note that the condition over the holonomy and the 
number of separatrices allow us to apply the classic Mattei-Moussu's theorem (see \cite{M-M}) to the induced foliation at $\tilde{\F}_{U_0}$. Therefore, there exist a neighborhood 
$U_0'\subset U_0$, containing $p$, such that the leaves of $\tilde{\F}_{U_0'}$ are the level sets of some holomorphic function $f:U_0',p\to f(U_0'),0\subset \C$, then they 
are close off $p$.\\
Consider a set $V\subset\overline{V}\subset U_0'$ and for each leaf $L_i$ a relative open neighborhood $V_i\subset L_i$ of $p$ such that $V_i\subset L_i\cap V$ and for each $p_j$ a 
relative open neighborhood $\tilde{V}_j$. As $L_i\setminus( 
V_i\cup\tilde{V}_j)$ is compact, it can be cover by finite many trivializing chats, call $W_i$ the union of all of 
them together with the charts for the dicritical components mentioned above. Using the 
finiteness of the holonomy we can construct a fundamental system $\mathcal{W}_i$ of $\tilde{\F}_{W_i}$-saturated neighborhoods of ($W_i$-relatively) closed leaves. By construction, 
each fundamental system intersect $U_0'$ which is $\tilde{\F}_{U_0'}$-saturated by closed leaves, thus the leaves of $\tilde{\F}$ contained in the union of some 
representative of each $\mathcal{W}_i$ with $U_0'$ are compact in $\tilde{M}$. Therefore, there exist infinitely many compact leaves and according to Theorem \ref{teo_ghys} this 
implies that 
$\tilde{\F}$ has a meromorphic first integral, as $\tilde{\F}\setminus D$ and $\F\setminus\{p_j\}_j$ coincide, it is also meromorphic first integral of $\F$.    
\end{proof}
\chapter{First integrals around the separatrix set}\label{Chap6}
\section{Introduction}
One of the key stones in the theory of holomorphic foliations is the article \cite{M-M}, where is presented the following important result about the existence of holomorphic
first integrals,
\begin{theorem}\label{M&M}
Let $\F$ be a germ in $0\in\C^2$ of holomorphic foliation of codimension $1$. Suppose that:
\begin{enumerate}
\item $\mathrm{Sing}\,(\F)=\{0\}$.
\item There are only finite many separatices $S_k$.
\item The leaves  are closed off the origin.
\end{enumerate}
Then, there exist a neighborhood $V$ of $0$, such that $\F|_V$ has a holomorphic first integral.
\end{theorem}
Years latter in \cite{Moussu}, one of their authors revisited this result in order to create a new proof, a simpler and more geometric one. In the first part of this chapter we 
present this proof with the aim to emphasize that to start the construction we do not require a small neighborhood but its transversality with the separatrices.\par
The second part contains two minor results product of an unsuccessful attempt to give a proof of Theorem \ref{existence} repeating Moussu's technique \cite{Moussu}.
\section{Moussu proof of Theorem \ref{M&M}}
Throughout this chapter we identify $\C^2$ with $\mathbb{R}^{4}$ (together with the euclidean norm $\|\ \|$) and use the notation:
\begin{itemize}
\item $B$ for the open ball in $0$ of radius $r$ in $\mathbb{R}^4$, $\partial B$ the sphere of radius $r$ and $\overline{B}$ the closure of $B$.
\item $\F_r$, $\partial\F$ and $\overline{\F}$ for the foliations induced by $\F$ in $B$, $\partial B$ and $\overline{B}$ respectively.
\end{itemize} 
Note that $\partial\F$ is a foliation by curves (real dimension $1$) with singularities where the leaves of $\F$ are tangent to the sphere $\partial B$, and the 
intersection of the separatices of $\F$ with $\overline{B}$ is the union $S$ of $l$ irreducible curves $S_k$, $1\leq k\leq l$. The border $\partial S_k=S_k\cap\partial B$ of $S_k$ is 
a smooth curve homeomorphic to a circle and $S_k^*=S_k\setminus\{0\}$ is a leaf of $\overline{\F}$.\par 
The property needed to carry on this proof is the transversality of the separatrices with the sphere $\partial B_r$ but it is 
possible to consider a function $g$ on some neighborhood $U$ of $0$ with a Morse critical point at $0$ of index $0$, so that its non-critical levels are diffeomorphic to spheres and 
rewrite \cite{Moussu} using the level sets of $g$ instead of spheres.\par  
The proof is divided in two parts:
\begin{enumerate}[A.]
\item Construction a neighborhood $V$ of the origin.
\item Study of the quotient space $V/\F_V$.
\end{enumerate} 
The neighborhood $V$ of the origin would be analogous to a "Milnor neighborhood'', see \cite{sing_Mil} and \cite{MSV} for its definition and properties, and also see 
\cite{Sea-Lim} where the spheres used in the previous two references are replaced for level sets of a map $g$ as the one above.  
\begingroup 
\setcounter{tmp}{\value{theorem}}
\setcounter{theorem}{0} 
\renewcommand\thetheorem{\Alph{theorem}}

\begin{lem}\label{lemA}
Exist a neighborhood $V$ of $S$ in $\overline{B}$ such that, $V$ is $\overline{\F}$-invariant and the leaves in $V$ cut $\partial B$ transversally. 
\end{lem}
Let us set $V^*=V\setminus{S}$, $\F_{V^*}$ the foliation induced by $\F$ in $V^*$, $V^*/\F_{V^*}$ the quotient space and $q_{V^*}$ the quotient map ($q_{V^*}:V^*\rightarrow
V^*/\F_{V^*}$).
\begin{lem}\label{lemB}
Exist a homeomorphism \[h:V^*/\F_{V^*}\to\D^*(=\D_1\setminus\{0\})\] such that $h\circ q_{V^*}=p_{V^*}$ is holomorphic.
\end{lem}

\setcounter{theorem}{\thetmp}
\renewcommand\thetheorem{\arabic{theorem}}
\endgroup

\begin{proof}[Proof of Theorem \ref{M&M}]
  Therefore, $p_{V^*}$ is holomorphic and bounded, and $S$ is an analytic set of cod $1$ (see \cite{gunning1}), then $p_{V^*}$ extents as a holomorphic first integral in $V$.
\end{proof}
\subsection{Proof of the lemmas}
\begin{proof}[Proof of Lemma \ref{lemA}]
The proof follows from the following affirmations: 
\begin{aff}\label{aff1}
If $L$ is a leaf of $\F$ transverse to $\partial B$. Then, it exist a fundamental system of neighborhoods $\overline{\F}$-invariant of $L$ in $\overline{B}$.
\end{aff}
-If L is transverse to $\partial B$ using the fact that leaves in $\overline{\F}$ are compact off $0$ and by the Theorem \ref{fund-theo} (finite orbits $\Leftrightarrow$ periodicity) 
the
holonomy of $L$ is finite. We can use Reeb's Theorem in $(L,\F)$ showing the affirmation.\par 
Now, we have that for each $k$, the curve $S_k\cap\partial B$ possesses a neighborhood where $\partial\F$ is a transversally holomorphic foliation without singularities.
$S_k\cap\partial B$ is compact with finite holonomy, applying Reeb's in $(\partial S_k,\partial\F)$ we have:
\begin{aff}\label{aff2}
For $k=1,2,\dots,l$, the leaf $\partial S_k$ of $\partial\F$ possesses a tubular neighborhood $T_k(\epsilon)$ in $\partial B$\[J_k:\D_\epsilon\times S^1\to T_k(\epsilon),\] such that
$J_k^{-1}(\partial\F)$ is the suspension of a periodic rotation in $\D$.\par 
$T_k(\epsilon)$ is $\partial\F$-invariant and $T_k(\epsilon')=J_k(\D_{\epsilon'}\times S^1)$ with $0<\epsilon'<\epsilon$ forms a fundamental system of neighborhoods of $\partial S_k$
in $\partial B$. In addition $T_k(\epsilon)$ is transverse to $\F$.
\end{aff}
\begin{aff}\label{aff3}
It exist $0<\epsilon'<\epsilon$ such that the intersection of $\partial B$ with the $\overline{\F}$-saturated $V(\epsilon')$ of $T_1(\epsilon')$ is contained in $T(\epsilon)=\cup
T_k(\epsilon)$.
\end{aff}
-By contradiction, take a sequence $\{a_k\}_k$ of points in $T_1(\epsilon)$ such that $a_k\to a\in\partial S_1$ and satisfying $L_{a_k}\cap\partial B\not\subset T(\epsilon)$ where
$L_{a_k}$ is the leaf in $\overline{\F}$ passing by $a_k$. Take $b_k$ a point in $(L_{a_k}\cap\partial B)\setminus T(\epsilon)$, then $\{b_k\}_k$ is a sequence in a compact thus
$b_k\to b$ (using the same notation for a subsequence), if $L_b$ is transverse to $\partial B$ then by affirmation \ref{aff1} $L_b$ is far from $S_1$ ($\rightarrow\leftarrow$).\\
If $L_b$ is not transverse to $\partial B$ we can take a sphere of radius $1+\delta$, $\delta>0$, and apply affirmation \ref{aff1} again ($\rightarrow\leftarrow$).  
\begin{aff}\label{aff4}
It exist $0<\epsilon_1<\epsilon'$ such that $V(\epsilon_1)=V$, the $\overline{\F}$-saturate of $T_1(\epsilon_1)$, is a neighborhood of $0$ in $\overline{B}$.
\end{aff} 
-The pseudo-group of holonomy is generated by a enumerable set of biholomorphisms with finitely many non trivial fixed points. The set of leaves of $\overline{\F}$ with non-trivial
holonomy is numerable (see \cite{Godbillon} proposition 2.7, pag. 96) so we can choose $\epsilon_1$ such that $0<\epsilon_1<\epsilon'$ and the leaves cutting
$J_1(\partial\mathbb{D}_1\times\{1\})=C_{\epsilon_1}$ have trivial holonomy. Again, the compactness of the leaves allows to apply Reeb stability theorem. For all $a\in C_{\epsilon_1}$
the leaf $L_a$ in $\overline{\F}$ through $a$ possesses a $\overline{\F}$-saturated tubular neighborhood:
\[J_a:\tau_a\times L_a\to T(L_a),\]
such that $J^{-1}(\overline{\F})$ is foliated by fibers $z\times L_a$, where $\tau_a$ is a small curve transverse to $\overline{\F}$ through $a$ contained in $T_1(\epsilon')$. In
particular the $\overline{\F}$-saturate of $\nu_a=\tau_a\cap C_{\epsilon_1}$ is $C^{\infty}$-diffeomorphic to the product $\nu_a\times L_a$ and the saturated of $C_{\epsilon_1}$ is a
$C^{\infty}$-hypersurface (whose boundary is contained in $\partial B$) fibered over $S^1$. By construction, is the boundary of $V=V(\epsilon_1)$ the $\overline{\F}$-saturated of
$T_1(\epsilon_1)$.        
\end{proof}
\begin{rem}
The construction in the previous affirmations does not work with infinitely many separatrices, because infinitely many implies that all leaves are separatrices and for 
instance in Affirmation \ref{aff4} cannot be avoided find separatrices cutting $C_{\epsilon_1}$ (for a small $\epsilon_1$) and then its saturate does not bound a neighborhood of 
$0$. 
\end{rem}
\begin{proof}[Proof of Lemma \ref{lemB}]
We are going to endow $\tilde{\Delta}:=V^*/\F_{V^*}$ (leaves space) with a Riemann surface structure and show that $\tilde{\Delta}$ and $\D^*$ are biholomorphic.\par 
Note first that by Affirmation \ref{aff1}, $\tilde{\Delta}$ is Hausdorff, with the topology induced by the quotient map $q_V:V^*\to V^*/\F_{V^*}$. Note also, that $V^*$ is the
saturation of $T_1(\epsilon_1)$ which is the same that the saturation of $J_1(\D_{\epsilon_1}\times 1)=:\Delta^*_{\epsilon_1}$ and $\Delta^*_{\epsilon_1}$ is transverse to $\F_{V^*}$
and diffeomorphic to $\D^*$, in $\Delta^*_{\epsilon_1}$ we consider the topology induced by this diffeomorphism. With this in mind we can identify
\[V^*/\F_{V^*}=q_{V^*}(V^*)=q_{V^*}(\Delta^*_{\epsilon_1})\]\[ \tilde{\Delta}=q(\Delta^*_{\epsilon_1}).\]
Now, take $a\in\Delta^*_{\epsilon_1}$ and $L_a$ the leaf in $\F_{V^*}$ passing by $a$. Observe that $L_a$ is compact, transverse to $\partial B$ and with finite holonomy. This
holonomy is a subgroup of the group of rotations centered at $0\in\C$, and is isomorphic to $\mathbb{Z}/n(a)\mathbb{Z}$, with $n(a)\in\mathbb{Z}$. Applying Reeb's to ($L_a,\F_{V^*}$)
we find a neighborhood $\F_{V^*}$-invariant of $L_a$ that can be thought as the $\F_{V^*}$-saturated of $\Delta_a$, a neighborhood of $a$ in $\Delta_{\epsilon_1}^*$, where $\Delta_a$
is biholomorphic to $\D_1$ (again, in $\Delta_a$ we consider the induced topology), this neighborhood is biholomorphically conjugated to $\mathbb{D}_1\times L_a$, having a first
integral $z\to z^{n(a)}$.\par Therefore, there exist a biholomorphism $\varphi_a:\mathbb{D}_1\to\Delta_a$, and a homeomorphism $g_a:q(\Delta_a)\to\mathbb{D}_1$ such that
\begin{displaymath}
    \xymatrix{
        \mathbb{D}_1 \ar[r]^{\varphi_a}& \Delta_a\ar[d]^{q} \\
        &\ar[ul]^{g_a} q(\Delta_a)}
\end{displaymath}
Where $g_a$ can be defined by 
\begin{equation}\label{g_a} 
  g_a\circ q\circ \varphi_a (z)=z^{n(a)},
\end{equation}
We have that $\{g_a\}_{a\in\Delta^*_{\epsilon_1}}$ is an atlas
that define a differentiable structure in $\tilde{\Delta}$ (therefore, $\tilde{\Delta}$ is a real manifold of dimension two i.e. locally a surface ).
\begin{itemize}
\item $\tilde{\Delta}=\bigcup q(\Delta_a)\quad \checkmark$.
\item If $q(\Delta_a)\cap q(\Delta_b)\neq\emptyset$, $g_a\left(q(\Delta_a)\cap q(\Delta_b)\right)$ and $g_b\left(q(\Delta_a)\cap q(\Delta_b)\right)$ are open sets and 
$g_b\circ g_a^{-1}:g_a\left(q(\Delta_a)\cap q(\Delta_b)\right)\to g_b\left(q(\Delta_a)\cap q(\Delta_b)\right)$ is a biholomorphism $\checkmark$.
\end{itemize}
To see the second one, note that $q(\Delta_a)\cap q(\Delta_b)$ is intersection of two open sets and because $g_a$ is a homeomorphism $g_a\left(q(\Delta_a)\cap 
q(\Delta_b)\right)=\big[\varphi_a^{-1}\circ q_a^{-1}\big(q(\Delta_a)\cap q(\Delta_b)\big)\big]^{n(a)}$ is open. Finally, 
\[g_b\circ g_a^{-1}(\cdot)=(g_b\circ q)(g_a\circ q)^{-1}(\cdot),\]
and \ref{g_a} implies that $g_b\circ g_a^{-1}(\cdot)$ is a biholomorphism.\medskip\\
Then by construction the quotient map $q$ (in fact $q_{V^*}$) is holomorphic (because $g_a\circ q$ is holomorphic, we are thinking $\tilde{\Delta}$ as a manifold). Observe that
$q:\Delta^*_{\epsilon_1}\equiv \D^*\to\tilde{\Delta}$ is a branched lifting whose branching points correspond to the points $a$ such that $n(a)>1$. In addition, $q$ is proper then,
is a branched lifting with finitely many leaves.\par
Now, $\tilde{\Delta}$ can not be simply connected because in that case it would be biholomorphic to $\mathbb{D}_1$ or $\C$ and the preimage of its boundary $S^1$ (hyperbolic) or
$\{\infty\}$ (parabolic) necessarily has to be $\partial\mathbb{D}_1$ and $0$, which are of different kind. Therefore, $\tilde{\Delta}$ is not simply connected.\par
In order to show that $\pi_1(\tilde{\Delta})$ is generated by one element, take a point $q(a)$ and two different elements $\alpha,\,\beta\in\pi_1(\tilde{\Delta},q(a))$ and due to the
fact that $q$ is a finite covering, there exist $l,s\in\mathbb{Z}$ such that the lifting in $a$ of $\alpha^l,$ and $\beta^s$ are closed curves homotopic to the same generator of
$\pi(\D^*,a)$, then,  they are homotopic. Thus $\tilde{\Delta}$ is a surface with monogenous fundamental group and is homeomorphic to $\mathbb{D^*}$. 
If $B_1$, $B_2$ are the boundaries of $\tilde{\Delta}$ we have that $q^{-1}(B_i)$ is a boundary of $\mathbb{D}^*$ of the same class that $B_i$. Therefore $\mathbb{D}^*$ and
$\tilde{\Delta}$ are biholomorphic. 
\end{proof}
\section{Generic vector fields in dimension $n$}
This section is dedicated to show our attempt to prove Theorem \ref{existence} following the proof in the section above.
\subsection{Attempt to a geometric proof of Theorem \ref{existence}}
As before, we divided the proof in two parts: 
\begin{enumerate}[A'.]
\item Construction a neighborhood $V$ of the origin.
\item Study of the quotient space $V/\F_V$.
\end{enumerate} 
We succeeded to prove the first part i.e., we build a invariant neighborhood $V$ of the separatrices (in this case the distinguished axis and the dicritical hyperplane Proposition 
\ref{separatices}) that can be seen as the saturated of a transverse section to the distinguished axis. Is important to mention that this was already done in \cite{Reis} (Proposition 
1.) and unlike it we need the hypothesis of the leaves be closed, our proof is more geometric except by the implicid used of the following proposition   .\par
Fix a small enough ball $B=B_r^{2n}$ centered in $0\in\C^n (\cong \mathbb{R}^{2n})$ contained in an open set $U$ where the germ of generic vector field 
$\X\in\mathrm{Gen}(\mathfrak{X}(\C^n,0))$ is defined.
\begin{pro}\label{separatices}
 If $\X\in\mathrm{Gen}(\mathfrak{X}(\C^n,0))$ satisfies condition ($\star$) (see Definition \ref{star}) then the separatrices of $\F(\X)$ are $S_{\X}$ and the leaves contained in the 
dicritic hyperplane.  
\end{pro}
\begin{proof}
 Remember that a generic vector field can be written in the form \eqref{gen-vf}
  \[
	\X(x) = \lambda_1x_1 (1 + a_1(x))\frac{\partial}{\partial x_1}+\lambda_2x_2 (1 + a_2(x))\frac{\partial}{\partial x_2}+\lambda_3x_3 (1 + a_3(x))\frac{\partial}{\partial x_3},  
  \]  
  where $a_i\in\mathcal{M}_3$ for $i=1,2,3$, and it can be chosen $v$ such that Re$(\lambda_1/v),$ Re$(\lambda_2/v)<0$ e Re$(\lambda_3/v)>0$. Also, as 
$a_3(0)=0$ we know that for $|x|$ small $|a_3(x)|<\epsilon$ thus $|1+a_3(x)|\geq|1-|a_3(x)||\geq 1-|a_3(x)|>1-\epsilon$ and the function $\frac{1+a_i(x)}{1+a_3(x)}$ is 
holomorphic, take $1+\tilde{a}_i(x)=\frac{1+a_i(x)}{1+a_3(x)}$, suppose that $|\tilde{a}_i(x)|\leq\frac{\mathrm{Re}(\lambda_i/v)}{2|\lambda_i/v|}$ and write 
$\X$ like 
  \[
	\X(x) = \frac{\lambda_1}{v}x_1 (1 + \tilde{a}_1(x))\frac{\partial}{\partial x_1}+\frac{\lambda_2}{v}x_2 (1 + \tilde{a}_2(x))\frac{\partial}{\partial 
x_2}+\frac{\lambda_3}{v}x_3\frac{\partial}{\partial x_3},  
  \]  
  Now, if $\gamma(T)=(x_1(T),x_2(T),x_3(T))$ is a separatrix of $\F(\X)$ not contained in the hyperplane $x_3=0$ or in the $x_3$ axis. We know that $\gamma$ is 
$\F(\X)$-invariant then $\X(\gamma)=\gamma'$ which is equivalent to $x'_i(T)=(\lambda_i/v)x_i(T) (1 + \tilde{a}_i(\gamma(T)))$ for $i=1,2$ and $x'_3(T)=(\lambda_3/v)x_i(T)$. Consider 
the case where $T=t\in\mathbb{R}$ hence $\gamma(t)$ is a curve with real dimension one, take $\gamma(0)\neq 0$ as its initial point and $\lim_{t\to\infty}\gamma(t)=0$ therefore 
\begin{align*}
 x_i(t)=x_i(0)e^{\frac{\lambda_i}{v}t +\frac{\lambda_i}{v}\int_0^t\tilde{a}_i(\gamma(t))\mathrm{d}t}
\end{align*}
for $i=1,2$ and $x_3(t)=x_3(0)e^{\frac{\lambda_3}{v}t}$. Now, taking norms 
 \[
  |x_i(t)|=|x_i(0)|e^{\mathrm{Re}\big(\frac{\lambda_i}{v}\big)t +\mathrm{Re}\big(\frac{\lambda_i}{v}\int_0^t\tilde{a}_i(\gamma(t))\mathrm{d}t\big)},
 \]
 considering the upper quotes
\begin{align*}
   \mathrm{Re}\Big(\frac{\lambda_i}{v}\int_0^t\tilde{a}_i(\gamma(t))\mathrm{d}t\Big)&\leq \big|\frac{\lambda_i}{v}\big|\int_0^t|\tilde{a}_i(\gamma(t))|\mathrm{d}t,\\
   &\leq \frac{1}{2}|\mathrm{Re}(\lambda_i/v)|t,
\end{align*}
   we have that $|x_i(t)|\leq|x_i(0)|e^{\frac{1}{2}\mathrm{Re}(\lambda_i/v)t}$ and this goes to $0$ when $t\to\infty$. On the other hand 
$|x_3(t)|=|x_3(0)|e^{\mathrm{Re}(\frac{\lambda_3}{v})t}$ and goes to $\infty$ because Re$(\lambda_3/v)>0$. In conclusion, $\gamma$ can not be as we supposed and it has to be 
contained in the hyperplane $x_3=0$ or in the $x_3$ axis.

\end{proof}

\begingroup 
\setcounter{tmp}{\value{theorem}}
\setcounter{theorem}{0} 
\renewcommand\thetheorem{\Alph{theorem}'}

\begin{lem}\label{lemA-dim3}
There exists open sets $V$ with $V\subset\overline{B}$ such that, $V$ is a neighborhood $\overline{\F}$-invariant of $S$ (the union of separatices).
\end{lem}

\setcounter{theorem}{\thetmp}
\renewcommand\thetheorem{\arabic{theorem}}
\endgroup

\begin{proof}[Proof of Lemma \ref{lemA-dim3}]
In this paragraph we used some of the arguments of the proof of Lemma 2. 
in \cite{livCamNet} pag. 66. First observe that $\partial L=L\cap \partial B$ is a closed set of real dimension one, and each connected component in $\partial L$ is diffeomorphic to 
the circle $S^1$. Suppose that 
$K\subset\partial L$ is one of this connected components, consider neighborhoods $U_k\supset W_k$ of $K$, $U_K$ open in $\C^n$ and $W_K$ open in $L$, where $W_K$ can be taken as a 
finite union of plates because $K\subset L$ is a compact subset of a leaf. As $\partial B$ intersects $W_K$ transversally, we can choose $U_K$ small enough such that for every $x\in 
U_K$ the leaf of $\F|_{U_K}$ through $x$ meets $\partial B$ transversally.\\ 
Continuing with this argument, if there exist $K_1$ and $K_2$ as above, we can use the same technique of 
the construction of the holonomy map to show that there exit and homeomorphism between transversal sections to $W_{K_1}$ and $W_{K_2}$ contained respectively in $U_{K_1}$ and 
$U_{K_2}$. This homeomorphism shows that we can find an invariant neighborhood of $L$ of leaves transversal to $\partial B$ in $\partial B\cap U_{K_1}$ and 
$\partial B\cap  U_{K_2}$.\\
In what follows we will use the notation $K_1=S_{\X}\cap \partial B$ where $S_{\X}$ is the distinguished axis of the generic vector field $\X$, be $U_{K_1}$ as before 
take $T_1(\epsilon)\subset U_{K_1}$ a set containing $K_1$ and diffeomorphic to $B^4_{\epsilon_1}\times S^{1}$ ($J_1(B^4_{\epsilon_1}\times S^{1})=T_1(\epsilon)$) and 
\[T_2(\epsilon_2)=\{x\in\C^3\,\big|\, |x_1|^2+|x_2|^2=1,\ |x_3|\leq\epsilon_2\}\]
\begin{aff}\label{aff3-dim3}
It exist $0<\epsilon'<\epsilon$ such that the intersection of $\partial B$ with the $\overline{\F}$-saturated $V(\epsilon')$ of $T_1(\epsilon')$ is contained in $T(\epsilon)=
T_1(\epsilon)\cap T_2(\epsilon)$.
\end{aff}
-By contradiction, take a sequence $\{a_k\}_k$ of points in $T_1(\epsilon)$ such that $a_k\to a\in K_1$ and satisfying $L_{a_k}\cap\partial B\not\subset T(\epsilon)$ where
$L_{a_k}$ is the leaf in $\overline{\F}$ through $a_k$. Take $b_k$ a point in $(L_{a_k}\cap\partial B)\setminus T(\epsilon)$, then $\{b_k\}_k$ is a sequence in a compact thus
$b_k\to b\in\partial B$ (using the same notation for a subsequence), if $L_b$ is transverse to $\partial B$ then we can use the previous paragraph supposing that $b$ belongs to some 
$K_2$, then it exist an invariant neighborhood of $L_b$ of leaves transversal to $\partial B$ in $\partial B\cap U_{K_1}$ and $\partial B\cap  U_{K_2}$, this implies that $L_b$ is 
far from $S_{\X}$.  ($\rightarrow\leftarrow$).\\
If $L_b$ is not transverse to $\partial B$ we can take a sphere of radius $1+\delta$, $\delta>0$, and proceed as above.  
\begin{aff}\label{aff4-dim3}
It exist $0<\epsilon_1<\epsilon'$ such that $V(\epsilon_1)=V$, the $\overline{\F}$-saturate of $T_1(\epsilon_1)$, is a neighborhood of $0$ in $\overline{B}$.
\end{aff} 
-The pseudo-group of holonomy is generated by a enumerable set of biholomorphisms with finitely many non trivial fixed points. The set of leaves of $\overline{\F}$ with non-trivial
holonomy is numerable (see \cite{Godbillon} proposition 2.7, pag. 96) so we can choose $\epsilon_1$ such that $0<\epsilon_1<\epsilon'$ and the leaves cutting
$J_1(\partial B^4_{\epsilon_1}\times\{1\})=C_{\epsilon_1}$ have trivial holonomy and the compactness of the leaves allows to apply Reeb stability theorem. For all $a\in 
C_{\epsilon_1}$
the leaf $L_a$ in $\overline{\F}$ through $a$ possesses a $\overline{\F}$-saturated tubular neighborhood:
\[J_a:\tau_a\times L_a\to T(L_a),\]
such that $J^{-1}(\overline{\F})$ is foliated by fibers $z\times L_a$, where $\tau_a$ is a small curve transverse to $\overline{\F}$ through $a$ contained in $T_1(\epsilon')$. In
particular the $\overline{\F}$-saturate of $\nu_a=\tau_a\cap C_{\epsilon_1}$ is $C^{\infty}$-diffeomorphic to the product $\nu_a\times L_a$ and the saturated of $C_{\epsilon_1}$ is a
$C^{\infty}$-hypersurface (whose boundary is contained in $\partial B$) fibered over $S^1$. By construction, is the boundary of $V=V(\epsilon_1)$ the $\overline{\F}$-saturated of
$T_1(\epsilon_1)$.        
\end{proof}
We would like to have the analogous of Lemma \ref{lemB}, something like: 
\begin{quote}
\emph{"Lemma B': Exist a homeomorphism \[h:V^*/\F_{V^*}\to B^*(=B\setminus\{0\})\] such that $h\circ q_{V^*}=p_{V^*}$ is holomorphic."}
\end{quote}
In order to proof such lemma it would be necessary to understand the topology of the space of leaves $q(V^*)=V^*/\F$. We know that $q(V^*)=q\big(J_1(B_{\epsilon'}\times\{1\})\big)$ 
is a Hausdorff space (because the leaves we are considering are closed) but the big difference is that in dimension two it can be shown that the $q(V^*)$ is biholomorphic to $\D^*$, 
in our case would ${B^4}^*$, using machinery like the Riemann map and fundamental group which do not exist (or are not as useful) in greater dimension. Our 
intention of repeat Moussu's proof in dimension three was unsuccessful but it helped us to achieve a better understanding of our problem.
\begin{appendices}
\chapter{Algebraic properties of groups of diffeomorphisms}\label{demRib}
Here we give a sketch of the proof of Proposition \ref{ribón}. We start by introducing some notations, definitions and results needed for this purpose, they mainly come from    
\cite{Rib1}, we also recommend \cite{MarteloScá,MarteloRib}.  
\section{Preliminaries}
Given an element $\phi\in\mathrm{Diff}(\C^n , 0)$ we consider its action in the space
of $k$-jets. More precisely we consider the element $\phi_k\in\mathrm{GL}(\mathfrak{m}/\mathfrak{m}^{k+1})$ defined by 
\begin{align*}
  \mathfrak{m}/\mathfrak{m}^{k+1}&\overset{\phi_k}{\to}\mathfrak{m}/\mathfrak{m}^{k+1}\\
  g+\mathfrak{m}^{k+1}&\mapsto g\circ\phi+\mathfrak{m}^{k+1}
\end{align*}
where $\mathfrak{m}/\mathfrak{m}^{k+1}$ can be interpreted as a finite dimensional complex vector space. In this point of view diffeomorphisms are interpreted as
operators acting on function spaces.
\begin{defi}
 We define $D_k=\{\phi_k : \phi\in\mathrm{Diff}(\C^n , 0)\}$.
\end{defi}
The natural projections $\pi_{k,l} : D_k \to D_l$ for $k\geq l$ define a projective system and hence we can consider the projective limit $\varprojlim D_k$, it is the so called group 
of 
formal diffeomorphisms.
\begin{defi}
 Let $G$ be a subgroup of $\widehat{\mathrm{Diff}}(\C^n,0)$. We define $G_k$ as the smallest algebraic subgroup of $D_k$ containing $\{\varphi_k :\varphi\in G \}$. 
\end{defi}
\begin{defi}
 Let $G$ be a subgroup of $\widehat{\mathrm{Diff}}(\C^n,0)$. We define $\overline{G}^z$ as $\varprojlim_{k\in\mathbb{N}} G_k$, more precisely $\overline{G}^z$ is the subgroup 
of $\widehat{\mathrm{Diff}}(\C^n,0)$ defined by 
\[
  \overline{G}^z=\{\varphi\in\widehat{\mathrm{Diff}}(\C^n,0):\ \varphi_k\in G_k\ \forall\,k\in\mathbb{N} \}
\]
 We say that $\overline{G}^z$ is the pro-algebraic closure of $G$. We say that $G$ is \emph{pro-algebraic} if $G=\overline{G}^z$
\end{defi}
\begin{pro}\label{pro2.2}
  Let $\phi\in\widehat{\mathrm{Diff}}(\C^n,0)$. Then $\phi$ is unipotent if and
only if $j^1\phi$ is unipotent. 
\end{pro}
\begin{lem}\label{lem2.3}
 Let $H_k$ be an algebraic subgroup of $D_k$ for $k\in\mathbb{N}$. Suppose that $\pi_{l,k}(H_l)\subset H_k$ for all $l\geq k\geq 1$. Then $\varprojlim_{k\in\mathbb{N}}H_k$ is a 
pro-algebraic subgroup of $\widehat{\mathrm{Diff}}(\C^n,0)$. Moreover the natural map $\varprojlim H_j\to H_k$ is surjective for any $k\in\mathbb{N}$ if $\pi_{l,k}(H_l)=H_k$ for all 
$l\geq k\geq 1.$ 
\end{lem}
The group $G$ is a projective limit of algebraic groups and closed in the Krull topology by definition. Since $G_k$ is an algebraic group
of matrices and in particular a Lie group, we can define the connected component $G_{k,0}$ of the identity in $G_k$. We also consider the set $G_{k,u}$ of
unipotent elements of $G_k$.
\begin{pro}\label{pro2.6}
 Let $G$ be a subgroup of $\widehat{\mathrm{Diff}}(\C^n,0)$. Then we have $\overline{G}^z_0=\{\varphi\in\overline{G}^z\,:\,\varphi_1\in G_{1,0}\}$. Moreover $\overline{G}^z_0$ is 
pro-algebraic.
\end{pro}
\begin{rem}\label{rem2.8}
 Let $G$ be a solvable subgroup of $\mathrm{Diff}(\C^n,0)$. Since membership in $\overline{G}^z_0$ and $\overline{G}^z_u$ can be checked out in the first jet, these groups
 have finite codimension in $\overline{G}^z$. Indeed the kernels of the natural maps 
 \[
   \overline{G}^z\to G_1/G_{1,u}\text{ and }\overline{G}^z\to G_1/G_{1,0}
 \]
 are equal to $\overline{G}^z_u$ and $\overline{G}^z_0$ respectively by Propositions \ref{pro2.2} and \ref{pro2.6}. In particular $\overline{G}^z/\overline{G}^z_0$ is a finite group.
\end{rem}
\begin{pro}[Proposition 2. \cite{MarteloRib}]\label{pro2.7}
 Let $G\subset\widehat{\mathrm{Diff}}(\C^n,0)$ be a group. Then $\mathfrak{g}$ is equal to $\{\X\in\hat{\mathfrak{X}}(\C^n,0):\mathrm{exp}(t\X)\in\overline{G}^z\ \forall t\in\C\}$ 
and $\overline{G}^z_0$ is generated by the set $\{\mathrm{exp}(\X):\X\in\mathfrak{g}\}$. Moreover if $G$ is unipotent then the map 
\[
  \mathrm{exp}:\mathfrak{g}\to\overline{G}^z
\]
is a bijection and $\mathfrak{g}$ is a Lie algebra of nilpotent formal vector fields. 
\end{pro}
\begin{rem}\label{rem2.11}
 Invariance properties typically define pro-algebraic groups. Let us present an example. Consider $f_1,\dots,f_n\in\hat{\mathcal{O}}_n$ and 
 \[
 G=\{\varphi\in\widehat{\mathrm{Diff}}(\C^n,0)\,|\quad f_j\circ\varphi\equiv f_j\quad\forall\,1\leq j\leq n\}
\]
We define
\[
  H_k = \{A\in D_k : A(f_j + \mathfrak{m}^{k+1}) = f_j + \mathfrak{m}^{k+1}\ \forall 1\leq j\leq p\}
\]
 for $k\in\mathbb{N}$. It is clear that $H_k$ is an algebraic subgroup of $D_k$ for $k\in\mathbb{N}$. Moreover we have $\pi_{l,k}(H_l)\subset H_k$ for $l\geq k\geq 1$. Since 
$f\circ\phi-f = 0$ is equivalent to $f\circ\phi-f\in\mathfrak{m}^k$ for any $k\in\mathbb{N}$, the group $\varprojlim H_k$ is equal to $G$. Moreover $G$ is pro-algebraic by Lemma 
\ref{lem2.3}.
\end{rem}

\section{proof of Proposition \ref{ribón}}
\begin{pro}\label{rib_orig}
Let us consider $n$ elements $f_1,\dots,f_n$ of the field of fractions of $\hat{\mathcal{O}}_n$. Suppose $\mathrm{d}f_1\wedge\dots\wedge \mathrm{d}f_n\not\equiv 0$. Then the group
\[
 G=\{\varphi\in\widehat{\mathrm{Diff}}(\C^n,0)\,|\quad f_j\circ\varphi\equiv f_j\quad\forall\,1\leq j\leq n\}
\]
is finite.
\end{pro}
\begin{proof}
 We have that $G$ is pro-algebraic by Remark \ref{rem2.11}. Consider an element $\X=\sum_{j=1}^n aj\partial/\partial x_j$ in the Lie algebra $L(G)$ of $G$. By definition we have 
\[
    f_j\circ\mathrm{exp}(t\X)\equiv f_j\ \forall t\in\C\Longrightarrow \X(f_j)=\lim_{t\to 0}\frac{f_j\circ\mathrm{exp}(t\X)-f_j}{t}\equiv 0
\]
for any $1\leq j\leq n$. The property $\X(f_j)=0$ for any $1\leq j\leq n$ is equivalent to
\[
      \begin{pmatrix}
	\frac{\partial f_1}{\partial x_1}&\frac{\partial f_1}{\partial x_2}&\cdots&\frac{\partial f_1}{\partial x_n}\\
	\frac{\partial f_2}{\partial x_1}&\frac{\partial f_2}{\partial x_2}&\cdots&\frac{\partial f_2}{\partial x_n}\\
	\vdots&\vdots&\ddots&\vdots\\
	\frac{\partial f_n}{\partial x_1}&\frac{\partial f_n}{\partial x_2}&\cdots&\frac{\partial f_n}{\partial x_n}
      \end{pmatrix}  
      \begin{pmatrix}
	a_1\\ a_2\\ \vdots\\ a_n
      \end{pmatrix}=
       \begin{pmatrix}
	0\\ 0\\ \vdots\\ 0
      \end{pmatrix}.
\]
Since $\mathrm{d}f_1\wedge\dots\wedge \mathrm{d}f_n\not\equiv 0$, the $n\times n$ matrix in the previous equation has
a non-vanishing determinant and then $\X\equiv 0$. Hence $L(G)$ is trivial and $\overline{G}_0^Z$ is the trivial group by Proposition \ref{pro2.7}. Since $G/\overline{G}^z_0$ is 
finite by Remark \ref{rem2.8}, $G$ is finite.
\end{proof}
\end{appendices}    
\bibliographystyle{plain}
\def\cprime{$'$} \def\cprime{$'$}

\end{document}